\documentclass[12pt]{article}

\usepackage{graphicx} 
\usepackage{subcaption}
\usepackage[margin=1in]{geometry}
\usepackage{amsfonts,amsmath, amsthm}
\usepackage{cases,verbatim}
\usepackage{textalpha}
\usepackage{color, hyperref}
\usepackage{enumitem}

\newcommand{\X}{X}
\newcommand{\Y}{Y}
\newcommand{\truesol}{u^\dagger}    
\newcommand{\noisydata}{f^\delta}   
\newcommand{\op}{K}              

\newcommand{\fidelitynoisy}{H_{f^{\delta}}}
\newcommand{\de}{^{\dagger}}

\usepackage{algorithm}
\usepackage{algpseudocode}

\usepackage{pgfplotstable}
\usepackage{booktabs}
\usepackage{longtable}
\usepackage{pdflscape}
\usepackage{geometry}

\usepackage{xcolor}
\usepackage{colortbl}

\usepackage[maxbibnames=99, giveninits=true]{biblatex}

\renewbibmacro{in:}{}
\renewbibmacro*{issue+date}{%
  \setunit{\addcomma\space}
    \iffieldundef{issue}
      {\usebibmacro{date}}
      {\printfield{issue}%
       \setunit*{\addspace}%
       \usebibmacro{date}}
      \newunit}
\newtheorem{remark}{Remark}
\title{\textbf{Primal–dual methods and acceleration for Morozov and equality constrained regularization}}
\author{
Diana-Elena Mirciu\\
Department of Mathematics, University of Klagenfurt, Austria\\
\texttt{diana-elena.mirciu@aau.at}
\and
Martin Benning\\
Department of Computer Science, University College London, United Kingdom\\
\texttt{martin.benning@ucl.ac.uk}
\and
Elena Resmerita\\
Department of Mathematics, University of Klagenfurt, Austria\\
\texttt{elena.resmerita@aau.at}
}
\date{\today}

\DeclareMathOperator*{\argmin}{\arg \min}
\newcommand{\mb}[1]{\textcolor{blue}{Martin: #1}}


\theoremstyle{plain}
\newtheorem{theo}{Theorem}[section]
\newtheorem{lemma}[theo]{Lemma}
\newtheorem{example}[theo]{Example}
\newtheorem{prop}[theo]{Proposition}

\usepackage{listings}
\usepackage{xcolor}

\usepackage{float} 

\usepackage{pgfplotstable} 
\pgfplotsset{compat=1.18}

\usepackage{mathtools}
\mathtoolsset{showonlyrefs}

\usepackage{titlesec}

\titleformat{\section}
  {\normalfont\Large\bfseries}{\thesection}{1em}{}

\titleformat{\subsection}
  {\normalfont\large\bfseries\itshape}{\thesubsection}{1em}{}

\titleformat{\subsubsection}
  {\normalfont\normalsize\itshape}{\thesubsubsection}{1em}{}

\titleformat{\paragraph}[runin]
  {\normalfont\normalsize\bfseries\itshape}{\theparagraph}{1em}{}[.]

\titleformat{\subparagraph}[runin]
    {\normalfont\normalsize\itshape}{\thesubparagraph}{1em}{}[.]

\begin{document}

\maketitle
\begin{abstract}
This work develops a regularization analysis of non-accelerated and accelerated primal-dual methods for solving linear inverse problems in the presence of noisy data. We investigate a Condat-Vũ algorithm and an accelerated primal-dual hybrid gradient method in Hilbert spaces, with focus on quantifying the effect of data perturbations on the reconstruction error. For the non-accelerated scheme, we derive error estimates  in terms of Bregman distances, whereas for the accelerated scheme we establish error estimates in norm. The study accommodates a general class of convex data fidelities satisfying suitable perturbation conditions, which are verified explicitly for equality constrained and Morozov regularization. For the non-accelerated method, the analysis is further extended to Banach spaces, taking into account non-Euclidean geometries and including a particular non-reflexive setting tailored to  nonnegative solution reconstruction. The results recover known behavior in classical settings while extending the regularization analysis to these more general frameworks. Numerical experiments with representative regularizers, including sparsity and entropic models, support the theoretical findings and illustrate practical performance under noise.

\end{abstract}


\section{Introduction}
In many applications across imaging sciences, medical diagnostics, and engineering, one faces the fundamental challenge of recovering an unknown, true signal or image $u\de \in \X$ from indirect measurements in $\Y$, where $X$ and $Y$ are Hilbert spaces. We denote by $f$ the perfect data and by $f^\delta$ the noisy observation. The inverse model is written as
\begin{equation} \label{eq:inverse_problem}
    \op u\de = f, \qquad \| f - f^\delta \| \le \delta.
\end{equation}
Here, $\op \colon \X \to \Y$ is a bounded linear forward operator. In realistic scenarios, the perfect data $f$ is unavailable, and one only records $f^\delta$.

Inverse problems are inherently ill-posed~\cite{engl1996regularization}, which means  that a direct inversion of the operator may strongly amplify measurement noise, yielding unstable and practically unreliable reconstructions. To overcome this instability, one employs regularization methods that stabilize the inversion process by incorporating a priori knowledge about the unknown signal, usually encoded by some regularizer $J$. Therefore, our work considers the following  problem for some generically denoted penalty $J$, 
\begin{equation}\label{original}
\inf\limits_{u\in \X}J(u) \text{ such that } \op u=f.
\end{equation}

The classical and most widely used approach is the variational Tikhonov regularization. In this framework, an approximate solution is computed as a minimizer of a functional that balances a data fidelity term, which enforces consistency with noisy measurements, against a regularizer $J$, which promotes desired structural properties such as smoothness or sparsity~\cite{benning2018modern}. Despite its success, a key difficulty in variational regularization is the appropriate choice of the regularization parameter, which often requires solving the underlying optimization problem multiple times,  thus dealing with computational challenges in the case of large-scale, high-dimensional problems.

Alternatively, one may consider the equality constraint $\op u=f^\delta$ and solve the resulting problem using an efficient iterative scheme. In this case, iteration must be stopped  early enough to prevent convergence to inaccurate reconstructions caused by fitting noisy data \cite{molinari2024iterative}.

Morozov regularization (residual method) \cite{grasmair2011residual} solves the operator equation using the inequality constraint $\|Ku-f^\delta\|\leq\delta$. It provides an attractive alternative to Tikhonov regularization, since it  incorporates a residual bound determined by the noise level, instead of requiring the choice of a regularization parameter. The reader is referred to \cite{lorenz2013necessary} for a detailed comparison of several regularization strategies, including these two. Combined with an efficient iterative method, Morozov regularization can stably and accurately reconstruct the solution at lower computational cost than methods requiring repeated parameter selection.

Alongside variational methods whose designs include  explicit regularizers, iterative regularization exploits the stabilization effect of the iteration itself,  by terminating the process according to the noise level. The foundations of iterative regularization date back to the late 1970s, with \cite{bakushinskii1977methods,bakushinskii1979principle} developing  regularizing methods for monotone variational inequalities and \cite{thomas1979approximation} studying iterative approximations of generalized inverses. Comprehensive studies of the evolution of regularization theory, including iterative methods and their extensions to nonlinear problems and Banach space settings, can be found in the monographs \cite{engl1996regularization, kaltenbacher2008iterative, scherzer2009variational, schuster2012regularization, vogel2002computational, hansen2010discrete, ito2014inverse}.

The current work focuses on primal–dual splitting schemes for solving linear inverse problems. Although they have a  longer history, their recent developments in optimization and mathematical imaging have been strongly influenced by the primal-dual hybrid gradient (PDHG) method also known as the  Chambolle-Pock algorithm \cite{chambolle2011first}. It is a first-order splitting scheme for convex optimization problems with saddle-point formulation, which has been further developed by Condat and Vũ to include an additional smooth term with Lipschitz continuous gradient \cite{condat2013primal, vu2013splitting}. These procedures adjust the primal and dual variables separately, mainly requiring  evaluations of proximal mappings associated to possibly non-smooth functions and often amounting to closed-form updates.  Moreover, they perform operator splitting by decoupling the operator $K$ from the non-smooth terms appearing in the objective function. It is worth noting that, from an inverse problem perspective, the primal iterates approximate a desired solution, while the dual iterates approximate the source element appearing in the source condition and thus provide information relevant to convergence rate estimates.

\color{black}

More recently, the machine learning and optimization communities have shown increased interest in the implicit regularization properties of primal-dual methods, particularly for non-smooth and non-strongly convex penalties - \cite{matet2017don}, \cite{garrigos2018iterative}, \cite{calatroni2021accelerated}, \cite{villa2023implicit}, \cite{vega2024fast}, \cite{molinari2024iterative}, \cite{gao2025model}.  For further developments on primal–dual methods in related settings, we refer to \cite{hohage2014generalization}, \cite{valkonen2014primal}, \cite{jiang2023bregman}, \cite{huang2026adaptive}, \cite{darbon2021accelerated}, and to \cite{silveti2023stochastic} for a stochastic first-order primal–dual method. Note that the literature on primal-dual methods in the optimization community is vast, and a comprehensive review is beyond the scope of this work.


Our study extends previous works on the regularizing effect of certain primal-dual splitting methods, in order to stably deal with noisy data. It is devoted to (non-)accelerated primal–dual hybrid gradient  schemes in Hilbert spaces, as well as Condat–Vũ procedures adapted to non-Euclidean geometries in Banach spaces, such as simplex constraints via entropic mirror descent.
Although the considered methods and their analysis allow for general data fidelity terms, the complete theoretical results are established only for equality constrained and Morozov-type inequality constrained problems, where the required assumptions can be verified.

We next clarify the distinct roles played by the early stopping in the two constrained settings considered in this work.
In the equality constrained problem, convergence means exact fitting of the noisy data, so early stopping is essential for avoiding an unstable noisy limit. The situation is, however, somewhat different in the inequality constrained setting.
 For each fixed noise level, running the primal-dual procedures longer makes the iterates approach the solution of the corresponding noisy Morozov problem. That limit is already regularized, as known from the literature \cite{grasmair2011necessary}.
Independently, as the noise level tends to zero, the solutions of the Morozov problems approach a $J$-minimizing solution of the exact inverse problem. Therefore, fully solving each Morozov problem is theoretically stable. 
Our approach combines these two processes and stops the iterates  in order to balance  the error terms due to the optimization accuracy and the noise propagation. Therefore, in primal-dual schemes applied to the inequality constrained problem, early stopping is no longer the primary regularization mechanism, as in the equality constrained setting, but rather provides additional stability and improved accuracy.

Morozov constrained formulations have also been considered by \cite{haltmeier2024data} in the context of data-driven regularization, where a PDHG  scheme is used for the numerical solution, and by \cite{rigie2015joint} for multi-channel spectral CT reconstruction using a similar type of algorithm.

\subsection{Contributions}
Our main contributions are as follows. Firstly, we consider  the Condat-V\~u primal-dual splitting algorithm~\cite{condat2013primal, chambolle2011first} with noisy observations and general data fidelities verifying certain stability assumptions, and  derive convergence rates under the classical source condition. We establish these results first in a Hilbert space setting and subsequently extend them to general reflexive Banach spaces, as well as to a particular nonreflexive setting arising from entropic regularization, where  directional derivatives are used instead of gradients.  
Secondly, we study the accelerated primal-dual hybrid gradient algorithm in the presence of noisy data and obtain error estimates in Hilbert spaces. 
Finally, we validate these theoretical findings through numerical experiments and compare the observed behavior with the theoretical guaranties.

To our knowledge, our analysis is mainly complementary to the works \cite{molinari2024iterative}, \cite{calatroni2021accelerated}, and \cite{darbon2021accelerated}. On the one hand, \cite{molinari2024iterative} investigates the regularizing properties of primal–dual methods for the equality constrained problem under noisy data, but the analysis is restricted to the non-accelerated schemes. On the other hand,  \cite{calatroni2021accelerated} deals with acceleration under general data fidelities within a dual diagonal regularization framework, where optimization iterations employ a vanishing regularization parameter. Our approach instead applies accelerated PDHG directly to the primal–dual formulation, without introducing such an additional diagonal parameter. Last but not least, 
\cite{darbon2021accelerated}  establishes convergence for an accelerated variant of PDHG in Banach spaces within a non-Euclidean geometry and for a general class of convex data fidelities. However,  that analysis is carried out in reflexive spaces and considers only exact data.  

\subsection{Outline}


The remainder of this paper is organized as follows. Section~\ref{sect:Preliminaries} introduces the mathematical notation, defines the inverse problem setting, and recalls the necessary background on convex analysis. Section \ref{sect:particular-instances} shows ergodic convergence rates for the Condat-V\~u algorithm with noisy data in Hilbert spaces, while Section \ref{Acc-PDHG-section} analyzes an accelerated version of the PDHG method in a similar context. Examples of data fidelities that fit this framework are provided in Section \ref{sect:concrete-frameworks}.
Subsequently, Section \ref{sect:Condat-Vu-general} studies a Bregman version of the Condat-V\~u algorithm, tailored for reflexive Banach spaces, while Section \ref{ex:entropic_mirror_example} presents an entropic PDHG algorithm in a natural non-reflexive setting.
Additionally, in Section~\ref{sect:numerics}, we present numerical experiments, focusing on a sparsity problem and a hyperspectral abundance map recovery. Finally, Appendix~\ref{appendix:proofs} includes additional proofs and figures.

\section{Preliminaries}\label{sect:Preliminaries}
This section presents the notation, recalls the Fenchel duality background, and states a boundedness result for positive sequences, which is essential to proving  error estimates for the Condat-V\~u algorithm in a general setting of spaces.

Fix $p>1$ and assume the following:
    \begin{enumerate}[start=1,label={(A\arabic*)}]
    \item $\X, \Y$ are Hilbert spaces and the operator $\op :\X\rightarrow \Y$ is linear, bounded and ill-posed;\label{A1} 
    \item The data fidelities $H_f, H_{f^{\delta}}:\Y\rightarrow\mathbb{R}\cup\{+\infty\}$ are proper, convex and lower semicontinuous; \label{A2}
    \item The function $R:\X\rightarrow\mathbb{R}\cup\{+\infty\}$ is proper, convex and lower semicontinuous. \label{A3}
    \item The function $G:\X\rightarrow\mathbb{R}\cup\{+\infty\}$ is proper, convex, Fr\'echet differentiable and $L$-smooth, meaning that 
    $
        \forall u_1, u_2\in\operatorname{int\; dom}G, \quad D_G(u_1,u_2)\leq\frac{L}{p}\Vert u_1-u_2\Vert^p_\X,
    $ 
    where $D_G$ is the Bregman distance with respect to $G$. \label{A4}
    \item There exists $\bar{u}\in\operatorname{dom }(H_{f}\circ \op )\cap\operatorname{dom } R\cap\operatorname{dom } G$, such that $R+G$ is  continuous at $\bar{u}$.  \label{A5}
    \end{enumerate}

    In general, for a non-differentiable, convex function $g$, and two elements $x,y\in \operatorname{dom}g$ such that $\partial g(y)\neq\emptyset$, the Bregman distance with respect to $g$ and a subgradient $\xi\in\partial g(y)$ is given by $D_g(x,y) = g(x) - g(y) - \langle\xi, x-y\rangle.$ In addition, one can also work with the symmetric Bregman distance, 
    $D_g^s(x,y) := D_g(x,y)+D_g(y,x).$

We aim to solve 
\begin{equation}\label{intitial-pb-exact}
    \inf\limits_{u\in \X}\{ R(u)+G(u)\} \text{ such that } \op u=f,
\end{equation}
which can be equivalently formulated as 
\begin{equation}\label{optim}
    \inf\limits_{u\in \X}\{ H_{f}(\op u)+R(u)+G(u)\},
\end{equation}
with 
\begin{equation}\label{h_f-formula}
    H_f(z) = i_{\{f\}}(z).
\end{equation}
Here $i_C$ denotes the indicator function of a nonempty set $C$, i.e., it vanishes on $C$ and is $+\infty$ otherwise.

Although our study focuses on this specific choice of $H_f$, we formulate the assumptions for a general function $H_f$ to highlight the main ingredients underlying the proofs. We discuss in Remark \ref{remark:particular-cases} how the general assumptions simplify for the specific $H_f$ mentioned above.

Throughout this work, $u\de\in \X$ denotes a solution of problem \eqref{intitial-pb-exact}, for which both $R$ and $G$ are finite.
Due to  \eqref{h_f-formula}, one has the following inclusion for any $v\in \Y,$ 
\begin{gather}\label{incl-init}
    v \in\partial H_f(\op u\de).
\end{gather}

When deriving convergence rates for the considered algorithms, $u^\dagger$ is assumed to satisfy the source condition
 \begin{enumerate}[start=6,label={(A\arabic*)}]
    \item  $-\op ^*v\de\in\partial(R+G)(u\de), \text{ for some } v\de\in \Y.$ \label{source-condition}
\end{enumerate}
One can make the choice $v=v\de$ in \eqref{incl-init} to get 
\begin{gather}\label{opt-con}
    v\de \in\partial H_f(\op u\de),
\end{gather}
which is equivalent to 
\begin{gather}\label{OC1-gen-n}
    \op u\de\in\partial H_f^*(v\de).
\end{gather}

The Fenchel dual problem associated with \eqref{optim} has the formulation: 
    \begin{equation}\label{dualpb}
        \sup\limits_{v\in \Y}-\left\{ H^*_f(v)+(R+G)^* (-\op ^*v)\right\}.
    \end{equation}

Let $\mathcal{S}:\Y\times \Y \rightarrow[0,+\infty]$ be an appropriate discrepancy function in the data space with the property that $\mathcal{S}(y,z)=0$ if and only if $y=z.$ In practice, due to measurement errors, one does not record the exact data $f$, but only an approximation $f^{\delta}$. Therefore, we consider the following formulation for problem \eqref{optim} with noisy data,
\begin{equation}\label{noisy-data-optim-pb}
    \inf\limits_{u\in \X}\{ H_{f^{\delta}}(\op u)+R(u)+G(u)\},
\end{equation}
where $H_{f^{\delta}} = i_{C^{\delta}},$ with $C^{\delta} = \{ z : \mathcal{S}(f^{\delta},z)\leq \delta\}$.
 Although this setting includes the equality constrained problem for $\mathcal{S}(f^{\delta},\cdot)=i_{\{f^\delta\}}$, it is mainly known as the residual method, or Morozov regularization, studied in \cite{grasmair2011residual},
     \begin{equation}
    \inf\limits_{u\in \X}\{ R(u)+G(u)\} \text{ such that } \mathcal{S}(f^{\delta},\op u)\leq \delta.
\end{equation}

Of particular interest in this work is the instance $\mathcal{S}(f^{\delta},z) = \| f^{\delta} - z \|,$ which corresponds to the inequality constrained problem
\begin{equation}\label{morozov-reg}
    \inf\limits_{u\in \X}\{ R(u)+G(u)\} \text{ such that } \|\op u-f^{\delta}\|\leq\delta.
\end{equation}
To maintain generality, we first consider the framework employing a general data fidelity $H_{f^{\delta}}$, and subsequently verify that our assumptions hold for the choice $H_{f^{\delta}} = i_{C^{\delta}}$ discussed in Section \ref{sect:concrete-frameworks}. 
For technical reasons, throughout this work we quantify the noise level in terms of the norm of the data space. More precisely, we assume 
\begin{equation}\label{eq:noise_level}
     \| f - f^\delta \| \le \delta.
\end{equation}

The primal-dual formulation associated with the exact data problem reads as
\begin{align}\label{saddle-point-problem}
    \inf_{u\in \X} \sup_{v\in \Y} R(u) + G(u) + \langle v, \op u \rangle - H_f^*(v),
\end{align}
where the corresponding objective function 
$$\mathcal{L}:X\times Y\rightarrow\mathbb{R}\cup\{+\infty\},\qquad \mathcal{L}(u,v) := R(u) + G(u) + \langle v, \op u \rangle - H_f^*(v)$$ is  called the Lagrangian.
Note that  source condition \ref{source-condition} is equivalent to $(u\de, v\de)$ being a saddle point for  $\mathcal{L}$, i.e., for any $(u,v)\in X\times Y$, one has $\mathcal{L}(u\de, v)-\mathcal{L}(u, v\de)\leq0.$
\color{black}

As mentioned at the beginning of this section, we  provide a boundedness result which is essential to proving  error estimates for the Condat-V\~u algoritm in the Banach space setting, that  extends in some sense a lemma from \cite{schmidt2011convergence}. Its   proof can be found in Appendix \ref{appendix:proof-aux-result}. While the original version was designed for quadratic inequalities ($q=2$), our generalization holds for any order $q>1$. 

\begin{theo}\label{lemma-Schmidt-gen}
   Fix $p,q>1$ with $\frac{1}{p}+\frac{1}{q}=1$. 
 Let $(a_n)$ and $(\lambda_n)$ be  non-negative sequences, and $(S_n)$ a non-decreasing sequence verifying:
$
    \forall n\geq 1, a_n^q \le S_n + \sum_{k=1}^n \lambda_k a_k.
$
Then we have the following:
$
    \forall n\geq 1,  a_n \le \left(\frac{q}{q-1}S_n +\left(\sum_{i=1}^n \lambda_i\right)^p\right)^{1/q}.
$
\end{theo}

For the sake of completeness, we recall the  result from \cite{schmidt2011convergence} which shows a slightly different upper bound.

\begin{prop}[{\cite[12]{schmidt2011convergence}}, Lemma 1]\label{lemma-Schmidt}
    Let $(a_n)$ and $(\lambda_n)$ be two non-negative sequences and $(S_n)$ a non-decreasing sequence with $a_1\leq\frac{1}{2}\lambda_1 + \sqrt{S_1 +\frac{1}{4}\lambda_1^2} $, such that 
    \begin{equation}\label{lemma-hyp}
        \forall n\geq1,\quad a_n^2\leq S_n+ \sum\limits_{k=1}^n \lambda_k a_k.
    \end{equation}
    Then the following inequality holds:
    $
        \forall n\geq1, \quad a_n\leq\frac{1}{2}\sum\limits_{k=1}^n \lambda_k + \sqrt{S_n+\left( \frac{1}{2}\sum\limits_{k=1}^n \lambda_k \right)^2}.
    $
\end{prop}

\section{The Condat-V\~u algorithm with noisy data in the Hilbert space setting}\label{sect:particular-instances}

Error estimates for the Condat-Vũ algorithm in the exact data setting were established in \cite{chambolle2011first,chambolle2016ergodic}. To the best of our knowledge, corresponding estimates for noisy data have so far been obtained only in \cite{molinari2024iterative}, for the particular choice $H_{f^\delta}=i_{\{f^\delta\}}$ and for a slightly different variant of the non-accelerated algorithm investigated in \cite{chambolle2011first}. The minor distinction between the two non-accelerated schemes is discussed after presenting Algorithm \ref{Condat-Vu-al-4}.

This section shows ergodic convergence rates for the Condat-V\~u algorithm with general data fidelity functions in Hilbert spaces, in the event of noisy data.
The Hilbert space framework allows us to clearly present the main arguments before turning to the more technical proofs in Banach spaces - see  Section \ref{sect:Condat-Vu-general}. In this context, we assume that \ref{A4} holds for $p=2.$

To solve \eqref{optim}, one can derive the  Condat-V\~u algorithm from  conditions \ref{source-condition} and \eqref{OC1-gen-n} as a fixed-point procedure, namely

\begin{gather}\label{fpi-gen-n}
    \begin{pmatrix}
        0 \\ 0
    \end{pmatrix} \in
    \begin{pmatrix}
        \partial R\left(u^{k+1}\right)+ G'\left(u^{k+1}\right)+\op ^{*} v^{k+1} \\ \partial H_{f^{\delta}}^{*}\left(v^{k+1}\right)-\op  u^{k+1}
    \end{pmatrix}+\nabla B\left(u^{k+1}, v^{k+1}\right)-\nabla B\left(u^{k}, v^{k}\right),
\end{gather}
for some convex function $B$. Hereafter, we will use the shorthand $w = (u, v)$, $w^k = (u^k, v^k)$ for any $k\geq0$, and $w^{\dagger} = (u^{\dagger}, v^{\dagger})$. We will also write  $B(w)$ instead of $B(u, v)$, and denote its derivative by $\nabla B$, in order to distinguish functions of the joint variable $w$ from functions such as $G$, whose derivative is denoted by $G'$.

The function $B:\X\times \Y\rightarrow\mathbb{R}\cup\{+\infty\}$ is chosen as follows, for positive scalars $\tau, \sigma$,
\begin{equation}\label{B-fc-quadr}
    B(u, v)=\frac{1}{2 \tau}\Vert u\Vert^{2}+\frac{1}{2 \sigma}\Vert v\Vert^{2}-\langle v, \op u\rangle-G(u).
\end{equation}

The fixed-point iteration turns into
\begin{gather}\label{pre-alg}
        \begin{pmatrix}
        0 \\ 0
    \end{pmatrix} \in
    \begin{pmatrix}
        \partial R\left(u^{k+1}\right)+ G'\left(u^{k}\right)+\op ^{*} v^{k}+\frac{1}{\tau}u^{k+1}-\frac{1}{\tau}u^{k} \\ \partial H_{f^{\delta}}^{*}\left(v^{k+1}\right)-\op (2u^{k+1}-u^{k})+\frac{1}{\sigma}v^{k+1}-\frac{1}{\sigma}v^{k}
    \end{pmatrix},
\end{gather}
 yielding the Condat-V\~{u} described below.
\begin{algorithm}[H]
\caption{The Condat-V\~u algorithm for noisy data}
\label{alg:Condat-Vu-Hilbert}
\begin{algorithmic}[1]
\State \textbf{Input:} Initial values $u^0\in\operatorname{dom} G', v^0$, step sizes $\tau, \sigma > 0$, data $f^{\delta}$ 
        \For{$k = 0, 1, 2, \dots$}
        \State $u^{k+1} \gets \text{prox}_{\tau R} \left( u^{k} - \tau \left( \op^* v^{k} + G'(u^{k}) \right) \right)$
    
        \State $v^{k+1} \gets \text{prox}_{\sigma H_{f^{\delta}}^*} \left( v^{k} + \sigma \op  (2u^{k+1} - u^{k})) \right)$
        \EndFor
        \State \textbf{return} $u^{k+1}, v^{k+1}$
        \end{algorithmic}
\end{algorithm}
\vspace{-10pt}


To establish explicit error estimates in the presence of noisy data, stronger structural conditions are required.
To this end, we introduce two alternative assumptions, either of which is sufficient for the subsequent analysis. 
\begin{enumerate}[start=7,label={(A\arabic*)}]
    \item There exist a constant $C>0$ and subgradients $\zeta^{\delta} \in \partial H_{f^{\delta}}^{*}\left(v\de\right)$ and $\zeta \in \partial H_{f}^{*}\left(v\de\right)$ such that\label{A10}
    \begin{equation} \label{hypo-noisy}
        \forall k\in\mathbb{N}, \;\langle  \zeta^{\delta}-\zeta,v\de-v^k\rangle \leq C{\delta} \Vert v\de-v^k\Vert;
    \end{equation} 
    \item There exists a constant $C>0$ such that \label{redassump0}
    \begin{equation*} 
        \forall k\in\mathbb{N}, \;    H_{f^{\delta}}^*(v\de)-H_{f^{\delta}}^*(v^{k})+H_{f}^*(v^{k})-H_f^*(v\de) \leq C\delta\|v^{k}-v\de\|.
    \end{equation*} 
\end{enumerate}
\vspace{-5pt}
In Section \ref{sect:concrete-frameworks} we will discuss some cases when these inequalities are satisfied.

\begin{remark}\label{remark:particular-cases}
    If $H_f$ is given by \eqref{h_f-formula}, the inequalities from assumptions \ref{A10} and \ref{redassump0} become
    \begin{equation*}
        \forall k\in\mathbb{N}, \;\langle  \zeta^{\delta}-f,v\de-v^k\rangle \leq C{\delta} \Vert v\de-v^k\Vert,
    \end{equation*}
    respectively, 
    \begin{equation*}
        \forall k\in\mathbb{N}, \;    H_{f^{\delta}}^*(v\de)-H_{f^{\delta}}^*(v^{k})+\langle f, v^{k}-v\de\rangle \leq C\delta\|v^{k}-v\de\|.
    \end{equation*}
\end{remark}

We provide below error estimates for the ergodic averages
$
    \overline{u}^n := \frac{1}{n}\sum_{k = 1}^n u^k,  \overline{v}^n := \frac{1}{n}\sum_{k = 1}^n v^k.
$

\begin{theo}\label{main-theo-part}
    Let $\tau, \sigma>0$ be fixed. Suppose that assumptions \ref{A1} - \ref{A5} and \ref{A10} hold, and $\tau(\sigma\Vert \op \Vert^2+L)<1$.
    Then there exists $\nu > 0$ such that for any $n\geq1$,
\begin{gather}\label{estim:non-acc}
 D_{R+G}(\overline{u}^n, u^\dagger) + D_{H_{f^{\delta}}^\star}\left(\overline{v}^n, v^\dagger\right)  \leq \frac{1}{n} D_B(w\de,w^{0}) + \frac{C^2}{\nu}\delta^2 (n+1) + C\delta \sqrt{\frac{2D_B(w\de,w^0)}{\nu}}.
\end{gather}
\end{theo}

\begin{proof}

Fix $n\geq1$ and  $0\leq k\leq n-1$. We structure the proof in several steps. \\
\textbf{Step 1: } We aim to rewrite the fixed-point iteration to obtain an equality in terms of the Bregman distances $D_R, D_G, D_{H_{f^{\delta}}^*}$ evaluated at iteration ${k+1}.$

By employing \ref{source-condition} and \eqref{OC1-gen-n}, iteration \eqref{fpi-gen-n} has the following form,

\begin{gather}\label{fixed-point-iteration-part}
    \begin{pmatrix}
        0 \\ 0
    \end{pmatrix} \in
    \begin{pmatrix}
        \partial R\left(u^{k+1}\right)-\partial R\left(u\de\right)+G'\left(u^{k+1}\right)-G'\left(u\de\right)+\op^{*} (v^{k+1}-v\de) \\ \partial H_{f^{\delta}}^{*}\left(v^{k+1}\right)-\partial H_f^{*}\left(v\de\right)-\op (u^{k+1}-u\de)
    \end{pmatrix}+\nabla B\left(w^{k+1}\right)-\nabla B\left(w^{k}\right).
\end{gather}

Taking the dual product with $w^{k+1} - w^\dagger$, we obtain
\begin{gather*}
    0 = D_R^{s}(u^{k+1}, u^\dagger)+D_G^{s}(u^{k+1}, u^\dagger) + \left\langle  \partial H_{f^{\delta}}^{*}\left(v^{k+1}\right)-\partial H_f^{*}\left(v\de\right),v^{k+1}-v\de\right\rangle +  \\ \langle \nabla B(w^{k+1}) - \nabla B(w^{k}), w^{k+1} - w^\dagger \rangle. 
\end{gather*}
One can prove that $B$ is jointly (strictly) convex by showing that $\nabla B$ is a (strictly) monotone operator. This requires assumption \ref{A4} and inequality $\tau(\sigma\|K\|^2+L)<1$ - for details, see Section \ref{sect:Condat-Vu-general} which covers the reflexive Banach space setting.
Further, the three-point identity for $B$ implies
\begin{gather}
D_R^{s}(u^{k+1}, u^\dagger)+D_G^{s}(u^{k+1}, u^\dagger) + D_{H_{f^{\delta}}^\star}^{s}(v^{k+1}, v^\dagger) =\nonumber \\ 
    \left\langle  \partial H_{f^{\delta}}^{*}\left(v\de\right)-\partial H_f^{*}\left(v\de\right),v\de-v^{k+1}\right\rangle  - D_B(w\de,w^{k+1})-D_B(w^{k+1},w^{k})+D_B(w\de,w^{k}).\label{main-eq-n}
\end{gather} 

\textbf{Step 2: }We need an upper bound for $\|v\de-v^n\|$. Since this cannot be derived directly, we obtain in this step just a recursive inequality for $\|v\de-v^n\|^2$. 
Note that the left-hand side of \eqref{main-eq-n} is non-negative, and thus,
\begin{equation*}
    \left\langle  \partial H_{f^{\delta}}^{*}\left(v\de\right)-\partial H_f^{*}\left(v\de\right),v\de-v^{k+1}\right\rangle - D_B(w\de,w^{k+1})-D_B(w^{k+1},w^{k})+D_B(w\de,w^{k})\geq0. 
\end{equation*}
By summing up for $k$ from $0$ to $n-1$, one gets
\begin{equation*}
   \sum\limits_{k=1}^n{\left\langle  \partial H_{f^{\delta}}^{*}\left(v\de\right)-\partial H_f^{*}\left(v\de\right),v\de-v^k\right\rangle} -D_B(w\de,w^n)+D_B(w\de,w^0) -\sum\limits_{k=1}^n{D_B(w^k,w^{k-1})}\geq 0,
\end{equation*}
which, together with 
    $
        \sum\limits_{k=1}^n{D_B(w^k,w^{k-1})}\geq 0
    $ 
yields
\begin{equation} \label{ineq-Bregman-n}
    \sum\limits_{k=1}^n{\left\langle  \partial H_{f^{\delta}}^{*}\left(v\de\right)-\partial H_f^{*}\left(v\de\right),v\de-v^k\right\rangle} +D_B(w\de,w^0)\geq D_B(w\de,w^n).
\end{equation}
The right-hand side can be rewritten as follows,
\begin{gather*}
    D_B(w\de,w^n)=\frac{1}{2\tau}\Vert u\de-u^n\Vert^2 + \frac{1}{2\sigma}\Vert v\de-v^n\Vert^2 - D_G(u\de,u^n) + \langle v\de-v^n, \op (u^n-u\de)\rangle.
\end{gather*}
    Since $G$ is $L$-smooth, one has 
     $
        D_G(u\de,u^n)\leq\frac{L}{2}\Vert u\de-u^n\Vert^2.
     $
    One assumption  of the theorem implies $\sigma < \frac{1-L\tau}{\tau\|\op \|^2}$. Hence, applying Young's inequality with $\bar{\epsilon} \in \left( \sigma, \frac{1-L\tau}{\tau\|\op \|^2} \right)$ leads to
    \begin{gather*}
        \langle v\de-v^n, \op (u^n-u\de)\rangle \geq -\Vert v\de-v^n\Vert\cdot\Vert \op \Vert\cdot\Vert u\de-u^n\Vert \geq 
        -\frac{\Vert v\de-v^n\Vert^2}{2\bar{\epsilon}}-\frac{\bar{\epsilon}\Vert \op \Vert^2\cdot\Vert u\de-u^n\Vert^2}{2},
    \end{gather*}
    which further ensures
    \begin{gather}\label{inq-appendix}
        D_B(w\de,w^n)\geq \left(\frac{1}{\tau}-L - \bar{\epsilon}\Vert \op \Vert^2\right)\frac{\Vert u\de-u^n\Vert^2}{2} + \left( \frac{1}{\sigma}-\frac{1}{\bar{\epsilon}}\right)\frac{\Vert v\de-v^n\Vert^2}{2}.
    \end{gather}
     Denote $\nu = \frac{1}{\sigma} - \frac{1}{\bar{\epsilon}}$.
     With this choice, both coefficients of the right-hand side in \eqref{inq-appendix} are positive.
    This implies
    $
            D_B(w\de,w^n)\geq\nu\frac{\Vert v\de-v^n\Vert^2}{2},
    $
    which together with \eqref{ineq-Bregman-n} yields
    \begin{equation*}
        \Vert v\de-v^n\Vert^2 \leq \frac{2}{\nu}\sum\limits_{k=1}^n{\left\langle  \partial H_{f^{\delta}}^{*}\left(v\de\right)-\partial H_f^{*}\left(v\de\right),v\de-v^k\right\rangle} +\frac{2}{\nu} D_B(w\de,w^0).
    \end{equation*}
    By employing \eqref{hypo-noisy}, one gets
    \begin{equation}\label{estim-dual-CV}
        \Vert v\de-v^n\Vert^2 \leq \frac{2C\delta}{\nu}\sum\limits_{k=1}^n{\Vert v\de-v^k\Vert} +\frac{2}{\nu} D_B(w\de,w^0).
    \end{equation}
\textbf{Step 3:} We apply Proposition \ref{lemma-Schmidt}
 to get an estimate for $\|v\de-v^n\|$.
 
We make the following choices: $a_n := \Vert v\de- v^n\Vert\geq0$, $S_n := \frac{2}{\nu}D_B(w\de,w^0) \text{ non-decreasing}$, $\lambda_n := \frac{2C{\delta}}{\nu} \geq 0$
for any $n\geq 1$. One needs to show that $a_1\leq \frac{1}{2}\lambda_1+\sqrt{S_1+\frac{1}{4}\lambda_1^2}$, which is equivalent to 
\begin{equation*}
    \|v\de-v^1\|\leq \frac{C\delta}{\nu} + \sqrt{\frac{2}{\nu}D_B(w\de,w^0) + \frac{C^2\delta^2}{\nu^2}}.
\end{equation*}
It suffices to justify that 
\begin{equation*}
     \|v\de-v^1\|^2 - \frac{2C\delta\|v\de-v^1\|}{\nu} + \frac{C^2\delta^2}{\nu^2} \leq \frac{2}{\nu}D_B(w\de,w^0) + \frac{C^2\delta^2}{\nu^2}, 
\end{equation*}
which is exactly \eqref{estim-dual-CV} for $n=1$.
Then, for every $k$, 
\begin{gather} 
    \Vert v\de- v^{k+1}\Vert\leq \frac{\delta (k+1) C}{\nu}  + \sqrt{\frac{2D_B(w\de,w^0)}{\nu} + \left( \frac{\delta (k+1) C}{\nu}  \right)^2} \leq 
     \frac{2\delta (k
    +1)C}{\nu}  + \sqrt{\frac{2D_B(w\de,w^0)}{\nu}},\label{estim-dualvar}
\end{gather}
where  the last inequality is a consequence of $\sqrt{a+b}\leq\sqrt{a}+\sqrt{b},$ for $a, b$ non-negative.\\
\textbf{Step 4: }We plug the previous estimate in \eqref{main-eq-n} to get the conclusion.

 More precisely, by employing inequalities \eqref{hypo-noisy} and \eqref{estim-dualvar} in \eqref{main-eq-n}, one finds
\begin{gather}
    D_R^{s}(u^{k+1}, u^\dagger)+D_G^{s}(u^{k+1}, u^\dagger) + D_{H_{f^{\delta}}^*}^{s}(v^{k+1}, v^\dagger) \leq \nonumber \\ C \delta \left(\frac{2\delta (k+1) C}{\nu} + \sqrt{\frac{2D_B(w\de,w^0)}{\nu}}\right) - D_B(w\de,w^{k+1})-D_B(w^{k+1},w^{k})+D_B(w\de,w^{k}).
\end{gather}
Summing up from $0$ to $n-1$ and using the sum rule for the subdifferential imply
\begin{gather*}
    \sum\limits_{k=1}^{n}{\left(D_{R+G}^{s}(u^k, u^\dagger) + D_{H_{f^{\delta}}^*}^{s}(v^k, v^\dagger)\right)} \leq
    D_B(w\de,w^{0}) +  \frac{C^2\delta^2 n(n+1)}{\nu} + C\delta n \sqrt{\frac{2D_B(w\de,w^0)}{\nu}} .
\end{gather*}
Since  convexity  holds for the one-sided Bregman distances 
$D_{R+G}(u^k, u^\dagger)$ but not necessarily for their symmetric counterparts $D_{R+G}^s(u^k, u^\dagger)$,   we can  deduce the announced  estimates for the ergodic averages $\overline{u}^n = (\sum_{k = 1}^n u^k)/n$ and $\overline{v}^n = (\sum_{k = 1}^n v^k)/n$,
\begin{gather*}
 D_{R+G}(\overline{u}^n, u^\dagger) + D_{H_{f^{\delta}}^*}\left(\overline{v}^n, v^\dagger\right)  \leq \frac{1}{n} D_B(w\de,w^{0}) + \frac{C^2}{\nu}\delta^2 (n+1) + C\delta \sqrt{\frac{2D_B(w\de,w^0)}{\nu}}.
\end{gather*} 
\end{proof}

\begin{remark}
    If $\delta = 0$ in \eqref{estim:non-acc}, we recover the estimate from \cite{chambolle2016ergodic}:
    $
         D_{R+G}(\overline{u}^n, u^\dagger) + D_{H_{f^{\delta}}^*}\left(\overline{v}^n, v^\dagger\right)  \leq \frac{1}{n} D_B(w\de,w^{0}). 
    $
\end{remark}

\begin{remark}
 If we choose the stopping index $n_{\delta}\sim\frac{1}{\delta}$, then  
    \begin{equation}\label{concl-estim-noisy-basic}
    D_{R+G}(\overline{u}^n, u^\dagger) + D_{H_{f^{\delta}}^*}\left(\overline{v}^n, v^\dagger\right)  = O(\delta).    
    \end{equation}
    This guarantees  convergence of the ergodic sequence $(\bar{u}^n)$ to the optimal solution $u\de$ in the sense of the Bregman distance, when $\delta\rightarrow0$.
    
\end{remark}

\begin{remark}
 By employing condition \ref{redassump0} instead of \ref{A10}, we can analogously derive upper bounds similar to those of \cite{molinari2024iterative} for $\mathcal{L}(\overline{u}^n, v\de) - \mathcal{L}(u\de, \overline{v}^n) = D_{R+G}(\overline{u}^n, u^\dagger) + D_{H_{f}^*}\left(\overline{v}^n, v^\dagger\right).$
 While the primal variable error estimate is the same whether Assumption \ref{A10} or Assumption \ref{redassump0} is imposed, the corresponding dual variable estimate depends on which of the two alternative assumptions is considered. In particular, under Assumption \ref{redassump0}, we obtain error estimates for $D_{H_{f^{\delta}}^*}$, that is more natural in the context of ill-posed problems.

\end{remark}

\begin{remark}
    If $H_{f^{\delta}} = i_{C^{\delta}},$ with $C^{\delta} = \{ z : \|f^{\delta}-z\|\leq \delta\}$, then $D_{H^*_{f^{\delta}}}(\bar{v}^n, v\de) = \delta D_{\|\cdot\|}(\bar{v}^n, v\de)$, which together with \eqref{concl-estim-noisy-basic} gives the boundedness of $D_{\|\cdot\|}(\bar{v}^n, v\de)$, for any $n$.
\end{remark}

\section{The accelerated PDHG algorithm with noisy data in the Hilbert space setting}\label{Acc-PDHG-section}
In this section, we establish convergence rates for the accelerated PDHG algorithm within Hilbert spaces in the case of noisy data and general data fidelity terms. Our analysis is inspired by  \cite{chambolle2011first}, which treats the case of general fidelities with exact data.

Let   $G=0$ in  \eqref{optim} and suppose that conditions \ref{A1} - \ref{A3}, \ref{A5} and \ref{redassump0}  hold.
We will work under the following standard assumption, which is instrumental in deriving the error estimates:
\begin{enumerate}[start=9,label={(A\arabic*)}]
    \item There exists $\gamma>0$ such that the function $R$ is $\gamma$-strongly convex, i.e., 
   \begin{equation*}
       \forall u_1 \in \X,\, \forall u_2\in\operatorname{dom } R, \, \forall \xi\in\partial R(u_2), \quad R(u_1) \geq R(u_2) + \langle\xi, u_1-u_2\rangle + \frac{\gamma}{2} \|u_1 - u_2\|^2.
       \end{equation*}\label{A7}
\end{enumerate}
Consider the following accelerated version investigated in \cite{chambolle2011first}, but  adapted here  for noisy data:

\begin{algorithm}[H]
\caption{Accelerated PDHG algorithm}
\label{acc-alg}
\begin{algorithmic}[1]
\State \textbf{Input:} $\tau_0, \sigma_0 > 0$ such that $\tau_0 \sigma_0 \|\op \|^2 < 1$, $\tau_{-1}\neq0$, $(u^0, v^0) \in \X \times \Y$, $\gamma > 0$, data $f^{\delta}$
\State \textbf{Initialization:} $\tilde{u}^0 \gets u^0$, $u^{-1} \gets u^0$
\For{$k = 0, 1, 2, \dots$}
    \State $v^{k+1} \gets \operatorname{prox}_{\sigma_k  H_{f^{\delta}}^*}(v^k + \sigma_k \op  \tilde{u}^k)$ 
    
    \State $u^{k+1} \gets \operatorname{prox}_{\tau_k R}(u^k - \tau_k \op ^* v^{k+1})$ 
    
    \State $\theta_k \gets 1 / \sqrt{1 + 2 \gamma \tau_k}$ 
    
    \State $\tau_{k+1} \gets \theta_k \tau_k$
    
    \State $\sigma_{k+1} \gets \sigma_k / \theta_k$
    
    \State $\tilde{u}^{k+1} \gets u^{k+1} + \theta_k (u^{k+1} - u^k)$ 
\EndFor
\State \textbf{Return:} $u^{k+1}, v^{k+1}$
\end{algorithmic}
\end{algorithm}

\begin{remark}
    Algorithm \ref{acc-alg} is closely related to the method studied in \cite{chambolle2016ergodic}, in which the dual variable is evaluated first at each iteration, followed by the update of the primal variable. Note, however, that the algorithm from \cite{chambolle2016ergodic} can also be reformulated so that the primal variable is computed first. This alternative formulation is discussed in Remark 5.3 of \cite{chambolle2016introduction} for exact data.

\end{remark}

We first prove several auxiliary lemmas, which will be used to establish the main theorem of this section.
We start by studying the standard, non-accelerated version, namely Algorithm \ref{alg:Condat-Vu-Hilbert} from Section \ref{sect:particular-instances}, following the ideas from \cite{chambolle2011first},
\begin{gather}\label{non-acc-alg}
    \begin{cases}
        v^{k+1}=\operatorname{prox}_{\sigma H_{f^{\delta}}^*}(v^k+\sigma \op \tilde{u}) \\
        u^{k+1}=\operatorname{prox}_{\tau R}(u^k-\tau \op^* \tilde{v}),
    \end{cases}
\end{gather}
for some $\tilde{u}\in \X, \tilde{v}\in \Y$ depending on the previous iterates.

\begin{lemma}\label{lemma-non-acc}
    Consider the non-accelerated algorithm \eqref{non-acc-alg}. Then the following inequality holds for any $k\geq0$ and  $\tilde{u}, \tilde{v}$ as above,
    \begin{align}
    0 \geq & H_{f^{\delta}}^*(v^{k+1}) - H_{f^{\delta}}^*(v\de) -\langle \op u\de, v^{k+1}-v\de \rangle + R(u^{k+1}) - R(u\de)  + \langle \op^*v\de, u^{k+1}-u\de \rangle + \nonumber\\
    &  \langle \op (u^{k+1}-\tilde{u}) , v^{k+1} - v\de\rangle - \langle \op (u^{k+1} - u\de), v^{k+1}-\tilde{v}\rangle + \frac{1}{\sigma}d_v +  \frac{1}{\tau} d_u  + \frac{\gamma}{2}\Vert u\de-u^{k+1}\Vert^2, \label{ineq-non-acc}
\end{align}
where $d_v = \frac{\Vert v^k-v^{k+1}\Vert^2}{2} + \frac{\Vert v\de-v^{k+1}\Vert^2}{2} - \frac{\Vert v\de-v^k\Vert^2}{2}$ and $d_u = \frac{\Vert u^k-u^{k+1}\Vert^2}{2} + \frac{\Vert u\de-u^{k+1}\Vert^2}{2} - \frac{\Vert u\de-u^k\Vert^2}{2}$.
\end{lemma}
\begin{proof}
   Let $k\geq0$ be arbitrary but fixed. Using equalities \eqref{non-acc-alg} and  strong convexity of $R$, one has
\begin{gather}
    \begin{cases}
        \frac{v^k-v^{k+1}}{\sigma}+\op \tilde{u} \in \partial H_{f^{\delta}}^*(v^{k+1})\nonumber \\
        \frac{u^k-u^{k+1}}{\tau}-\op^{*} \tilde{v} \in \partial R\left(u^{k+1}\right)\nonumber
    \end{cases} \Rightarrow\\
       \forall v\in \Y, \quad  H_{f^{\delta}}^*(v) - H_{f^{\delta}}^*(v^{k+1}) \geq \left\langle\frac{v^k-v^{k+1}}{\sigma}, v-v^{k+1}\right\rangle + \langle \op \tilde{u} , v-v^{k+1}\rangle, \label{Hconv}
\end{gather} 
\begin{equation}\label{Rstrconv}
\;\forall u\in \X, \quad  R(u)-R(u^{k+1}) \geq \left\langle\frac{u^k-u^{k+1}}{\tau}, u-u^{k+1}\right\rangle-\langle \op (u-u^{k+1}), \tilde{v}\rangle + \frac{\gamma}{2}\Vert u-u^{k+1}\Vert^2.
\end{equation}
Summing \eqref{Hconv} and \eqref{Rstrconv}, employing the three-point lemma for $\frac{\|\cdot\|^2}{2}$ and using $(u,v) = (u\de,v\de)$ give
\begin{align*}
    0 \geq & H_{f^{\delta}}^*(v^{k+1}) - H_{f^{\delta}}^*(v\de) + R(u^{k+1}) - R(u\de) + \langle \op \tilde{u} , v\de-v^{k+1}\rangle - \langle \op (u\de-u^{k+1}), \tilde{v}\rangle+\\
     & \frac{1}{\sigma}d_v + \frac{1}{\tau}d_u + \frac{\gamma}{2}\Vert u\de-u^{k+1}\Vert^2.
\end{align*}
By adding and subtracting  $\langle \op u\de,v\de\rangle$, $\langle \op u\de, v^{k+1} \rangle$, $ \langle \op^*v\de, u^{k+1} \rangle$, one gets the conclusion.
\end{proof}

Now we focus on the accelerated algorithm.
Our goal is to derive telescoping sums for appropriate terms by using an approach similar to that in Section 5 of \cite{chambolle2011first}.
\begin{lemma}
    Consider the accelerated Algorithm \ref{acc-alg}.
Then, for any $k\geq0$, one gets
    \begin{align}\label{recursive-ineq}
    \frac{\Delta_k}{\tau_k} \geq & \frac{\Delta_{k+1}}{\tau_{k+1}}  + \frac{2}{\tau_k}\langle \op (u^{k+1}-u^k), v^{k+1}-v\de\rangle - \frac{2}{\tau_{k-1}} \langle \op (u^{k}-u^{k-1}), v^{k}-v\de\rangle  - \nonumber \\
    &  - \frac{1}{\tau_{k-1}^2}\|u^k-u^{k-1}\|^2 + \frac{\Vert u^k-u^{k+1}\Vert^2}{\tau_k^2}  - \frac{2C\delta\|v^{k+1}-v\de\|}{\tau_k},
\end{align}
where 
$
    \Delta_k = \frac{\|v\de-v^k\|^2}{\sigma_k} + \frac{\|u\de-u^k\|^2}{\tau_k}.
$
\end{lemma}
\begin{proof}
Notice that Algorithm \ref{acc-alg} is obtained from \eqref{non-acc-alg} by considering dynamic steps $(\tau_k,\sigma_k)$ (instead of $(\tau,\sigma)$ in \eqref{non-acc-alg}) and by setting $\tilde{u} = u^k + \theta_{k-1}(u^k-u^{k-1})$, $\tilde{v} = v^{k+1}$. To prove \eqref{recursive-ineq}, we apply Lemma \ref{lemma-non-acc} for these choices.
Let $k\geq0$ be fixed. Inequality \eqref{ineq-non-acc} becomes
\begin{align}\label{ineqacc}
    \frac{\Delta_k}{2}  \geq & H_{f^{\delta}}^*(v^{k+1}) - H_{f^{\delta}}^*(v\de) -\langle \op u\de, v^{k+1}-v\de \rangle + R(u^{k+1}) - R(u\de)  + \langle \op^*v\de, u^{k+1}-u\de \rangle +  \nonumber \\
    & \langle \op (u^{k+1} - u^k - \theta_{k-1}(u^k-u^{k-1})) , v^{k+1} - v\de\rangle - \langle \op (u^{k+1} - u\de), v^{k+1}- v^{k+1} \rangle +\nonumber\\
    & \frac{\Vert v^k-v^{k+1}\Vert^2}{2\sigma_k} + \frac{\Vert v\de-v^{k+1}\Vert^2}{2\sigma_k}  + \frac{\Vert u^k-u^{k+1}\Vert^2}{2\tau_k} + \frac{\Vert u\de-u^{k+1}\Vert^2}{2\tau_k}  + \frac{\gamma}{2}\Vert u\de-u^{k+1}\Vert^2.
\end{align}
To obtain telescoping sums, we rewrite the terms on the right-hand side that involve the functions $H_{f^\delta}^*$ and $R$.
Recall  source condition \ref{source-condition} and   relation \eqref{opt-con} in this setting where $G=0$: 
\begin{align}\label{optim-cond-PDHG}
    \begin{cases}
        \op u\de \in \partial H_f^* (v\de)\\
        -\op^*v\de\in\partial R(u\de).
    \end{cases}
\end{align}

 Since they use the data fidelity  $H_f$ for exact data, we need to reformulate the first two rows of the right-hand side of \eqref{ineqacc} which contain $H_{f^{\delta}}$, as follows:
\begin{align*}
    & H_{f^{\delta}}^*(v^{k+1}) - H_{f^{\delta}}^*(v\de) -\langle \op u\de, v^{k+1}-v\de \rangle + R(u^{k+1}) - R(u\de)  + \langle \op^*v\de, u^{k+1}-u\de \rangle = \\
    &  H_{f}^*(v^{k+1}) - H_{f}^*(v\de)  -\langle \op u\de, v^{k+1}-v\de \rangle + R(u^{k+1}) - R(u\de)  + \langle \op^*v\de, u^{k+1}-u\de\rangle + \\ 
    & H_{f^{\delta}}^*(v^{k+1}) - H_{f^{\delta}}^*(v\de)  - H_{f}^*(v^{k+1}) + H_{f}^*(v\de) \color{black}\geq \frac{\gamma}{2}\Vert u^{k+1}-u\de\Vert^{2} -  C\delta\Vert v^{k+1}-v\de\Vert  .
\end{align*}
For the last inequality, we employed \ref{redassump0}, the strong convexity of $R$, and the convexity of $H_f^*$ for the corresponding subgradients from \eqref{optim-cond-PDHG}.
By substituting this in \eqref{ineqacc}, one has 
\begin{align*}
    \frac{\Delta_k}{2} \geq & \frac{1}{2 \sigma_{k}} \left\Vert v\de-v^{k+1}\right\Vert^{2}+\left(\frac{1}{2 \tau_{k}}+\gamma\right)\left\Vert u\de-u^{k+1}\right\Vert^{2}+ \langle \op (u^{k+1}-u^k), v^{k+1}-v\de\rangle \\
    &  - \theta_{k-1} \langle \op (u^{k}-u^{k-1}), v^{k+1}-v\de\rangle + \frac{\Vert v^k-v^{k+1}\Vert^2}{2\sigma_k} + \frac{\Vert u^k-u^{k+1}\Vert^2}{2\tau_k}  - C\delta\|v^{k+1}-v\de\|.
\end{align*}

We rewrite the following term by adding and subtracting $v^k$, and apply Young inequality,
\begin{align*}
    & \theta_{k-1} \langle \op (u^{k}-u^{k-1}), v^{k+1}-v\de\rangle 
    =  \theta_{k-1} \langle \op (u^{k}-u^{k-1}), v^{k+1}-v^k\rangle + \theta_{k-1} \langle \op (u^{k}-u^{k-1}), v^{k}-v\de\rangle\leq\\
    &  \frac{\Vert v^{k+1}-v^k\Vert^2}{2\sigma_k} + \frac{\theta_{k-1}^2\Vert \op \Vert^2 \|u^k-u^{k-1}\|^2\sigma_k}{2} + \theta_{k-1} \langle \op (u^{k}-u^{k-1}), v^{k}-v\de\rangle .
\end{align*}

We plug this into the previous inequality and denote by $r_k:=\langle \op (u^{k+1}-u^k), v^{k+1}-v\de\rangle - \theta_{k-1} \langle \op (u^{k}-u^{k-1}), v^{k}-v\de\rangle -\frac{\theta_{k-1}^2\Vert \op \Vert^2 \|u^k-u^{k-1}\|^2\sigma_k}{2} $ to get
\begin{align*}
    \frac{\Delta_k}{2} \geq & \frac{1}{2 \sigma_{k}}\left\Vert v\de-v^{k+1}\right\Vert^{2}+\left( \frac{1}{2 \tau_{k}}+\gamma\right)\left\Vert u\de-u^{k+1}\right\Vert^{2}+ r_k - \frac{\Vert v^{k+1}-v^k\Vert^2}{2\sigma_k}  +  \\
    & \frac{\Vert v^k-v^{k+1}\Vert^2}{2\sigma_k} + \frac{\Vert u^k-u^{k+1}\Vert^2}{2\tau_k}  - C\delta\|v^{k+1}-v\de\|,
\end{align*}
which becomes
\begin{align}\label{telescoping-sum-before-seq-choice}
    & \Delta_k \geq  \sigma_{k+1}\frac{1}{ \sigma_{k}}  \frac{\Vert v\de-v^{k+1}\Vert^{2}}{ \sigma_{k+1}} + \tau_{k+1} \color{black} \left(\frac{1}{\tau_{k}}+2\gamma\right)\frac{\Vert u\de-u^{k+1}\Vert^{2}}{\tau_{k+1}}  + 2 r_k  + \frac{\Vert u^k-u^{k+1}\Vert^2}{\tau_k}  - 2C\delta\|v^{k+1}-v\de\|.
\end{align}

We aim to get a telescoping sum for the terms $\frac{\Delta_k}{\tau_k}$. As described by Algorithm \ref{acc-alg}, choose $(\tau_k), (\sigma_k)\subset(0,\infty)$ such that $ \sigma_0\tau_0\|\op \|^2<1$ and
\begin{equation}\label{seq-choice}
    \frac{\sigma_{k+1}}{\sigma_k}  = \frac{\tau_{k+1}}{\tau_k} \left(1+2\gamma\tau_k\right) = \frac{\tau_k}{\tau_{k+1}}.
\end{equation}
This implies  $
    \tau_{k+1} = \frac{\tau_k}{\sqrt{1+2\gamma\tau_k}}
$
and transforms inequality \eqref{telescoping-sum-before-seq-choice} into
\begin{align*}
   & \Delta_k \geq  \frac{\tau_k}{\tau_{k+1}} \frac{\Vert v\de-v^{k+1}\Vert^{2}}{ \sigma_{k+1}}  + \frac{\tau_k}{\tau_{k+1}}  \frac{\Vert u\de-u^{k+1}\Vert^{2}}{\tau_{k+1}}  +2r_k + \frac{\Vert u^k-u^{k+1}\Vert^2}{\tau_k}  - 2C\delta\|v^{k+1}-v\de\|.
\end{align*}
By dividing both sides by $\tau_k>0$, and
substituting $\theta_{k-1} := \frac{\tau_k}{\tau_{k-1}}$ as in Algorithm \ref{acc-alg}, it follows that

\begin{align*}
    \frac{\Delta_k}{\tau_k} \geq & \frac{\Delta_{k+1}}{\tau_{k+1}}  + \frac{2}{\tau_k}\langle \op (u^{k+1}-u^k), v^{k+1}-v\de\rangle - \frac{2}{\tau_{k-1}} \langle \op (u^{k}-u^{k-1}), v^{k}-v\de\rangle  - \\
    &  - \frac{\tau_{k}}{\tau_{k-1}^2}\Vert \op \Vert^2 \|u^k-u^{k-1}\|^2\sigma_k + \frac{\Vert u^k-u^{k+1}\Vert^2}{\tau_k^2}  - \frac{2C\delta\|v^{k+1}-v\de\|}{\tau_k}.
\end{align*}
From \eqref{seq-choice} one has
\begin{equation}\label{prod-seq}
    \tau_k\sigma_k\|\op \|^2=\tau_0\sigma_0\|\op \|^2<1,
\end{equation} 
which, together with the previous inequality, yields the conclusion.
\end{proof}

The following lemma provides an intermediate estimate for $\|v\de-v^n\|^2$ and $\|u\de-u^n\|^2$, with a bound depending on the distance between the dual variables $v\de$ and $v^k,$ with $ 1\leq k\leq n$.
\begin{lemma}
For any $n\geq1$, the following inequality holds,
    \begin{equation} \label{estim-norm-ineq}
    \frac{1-\|\op \|^2 \sigma_n\tau_n }{\sigma_n\tau_n}\Vert v\de-v^n\Vert^2 + \frac{\Vert u\de-u^n\Vert^2}{\tau_n^2} \leq \frac{\Delta_0}{\tau_0}  + 2C\delta\sum_{k=1}^{n}\frac{\|v^{k}-v\de\|}{\tau_{k-1}}  .
    \end{equation}
\end{lemma}
\begin{proof}
    Let $n\geq1$ be fixed. By summing up \eqref{recursive-ineq} from $0$ to $n-1$ and using $u^{-1}=u^0$, one finds
\begin{align*} 
    & \frac{\Delta_0}{\tau_0} \geq  \frac{\Delta_{n}}{\tau_{n}}  + \frac{2}{\tau_{n-1}}\langle \op (u^{n}-u^{n-1}), v^{n}-v\de\rangle  + \frac{\Vert u^{n-1}-u^{n}\Vert^2}{\tau_{n-1}^2}  - 2C\delta\sum_{k=0}^{n-1}\frac{\|v^{k+1}-v\de\|}{\tau_k}.
\end{align*} 
By substituting  $\Delta_n$ and using Young inequality, one has
\begin{align*} 
    \frac{\Vert v\de-v^n\Vert^2}{\sigma_n\tau_n} + \frac{\Vert u\de-u^n\Vert^2}{\tau_n^2} & \leq 
     \frac{\Delta_0}{\tau_0}  + \|\op \|^2\| v^n-v\de\|^2  + 2C\delta\sum_{k=0}^{n-1}\frac{\|v^{k+1}-v\de\|}{\tau_k}.
\end{align*}
By regrouping the terms, the conclusion holds.
\end{proof}

The next lemma provides an upper bound for $\|v\de-v^n\|,$  $n\geq1$, which is independent of the previous iterates.

\begin{lemma}
Let $n\geq1$ and $\beta:=\frac{\sigma_0\tau_0}{1-\|\op \|^2\sigma_0\tau_0}$. The following estimate holds,
\begin{equation}\label{estim-seq-v}
     \Vert v\de- v^n\Vert\leq 2 C\delta\beta\sum_{k=1}^n\frac{1}{\tau_{k-1}}  + \sqrt{\beta\frac{\Delta_0}{\tau_0}} .
\end{equation}
\end{lemma}
\begin{proof}
Since the second term on the left-hand side of \eqref{estim-norm-ineq} is non-negative and $\sigma_n \tau_n = \sigma_0 \tau_0$ holds, one obtains 
\begin{equation}\label{ineq-dual-var}
    \Vert v\de-v^n\Vert^2 \leq \beta \frac{\Delta_0}{\tau_0}  + 2C\delta\beta\sum_{k=1}^{n}\frac{\|v^k-v\de\|}{\tau_{k-1}}.
\end{equation}

We employ Proposition \ref{lemma-Schmidt} by making the following choices: $a_n := \Vert v\de- v^n\Vert\geq0,$ $ S_n := \beta \frac{\Delta_0}{\tau_0} \text{ non-decreasing},$ $\lambda_n := 2C\delta\beta\frac{1}{\tau_{n-1}} \geq 0.$
 Thus, we need to check  $\Vert v^1-v\de\Vert\leq  C\delta\beta\frac{1}{\tau_0} + \sqrt{\beta\frac{\Delta_0}{\tau_0} +  C^2\delta^2\beta^2\frac{1}{\tau_0^2}} $.
It suffices to show 
$
    \left(\Vert v^1-v\de\Vert - C\delta\beta\frac{1}{\tau_0}\right)^2\leq \beta\frac{\Delta_0}{\tau_0} +  C^2\delta^2\beta^2\frac{1}{\tau_0^2},
$
thus recovering \eqref{ineq-dual-var} for $n=1$.

Then, for every $n\geq 1$, we have
\begin{align*}
     \Vert v\de- v^n\Vert\leq& C\delta\beta\sum_{k=1}^n\frac{1}{\tau_{k-1}}  + \sqrt{\beta\frac{\Delta_0}{\tau_0}  + C^2\delta^2\beta^2  \left(\sum_{k=1}^n\frac{1}{\tau_{k-1}}\right)^2} \leq 2 C\delta\beta\sum_{k=1}^n\frac{1}{\tau_{k-1}}  + \sqrt{\beta\frac{\Delta_0}{\tau_0}},
\end{align*}
where the last inequality follows from $\sqrt{a+b}\leq\sqrt{a}+\sqrt{b},$ for any $a,b\geq0.$
\end{proof}
The  theorem below shows estimates for the primal iterates.
\begin{theo}\label{theo-acc}
The following estimation holds,
\begin{align}\label{estimate-acc-alg}
        \Vert u\de-u^n\Vert^2 \sim  & \frac{1}{\gamma^2n^2}\left(\frac{\Vert v\de-v^0\Vert^2}{\sigma_0\tau_0} + \frac{\Vert u\de-u^0\Vert^2}{\tau_0^2} \right)+ \nonumber\\
        &\frac{1}{\gamma^2n^2}\left( C \gamma \sqrt{\beta\left(\frac{\Vert v\de-v^0\Vert^2}{\sigma_0\tau_0} + \frac{\Vert u\de-u^0\Vert^2}{\tau_0^2}\right)} \delta n^2+\frac{1}{2}C^2\beta \gamma^2 \delta^2 n^4 \right),
\end{align}
where $\beta = \frac{\sigma_0\tau_0}{1-\|\op \|^2\sigma_0\tau_0}, \gamma$ is a strong convexity constant for $R$ and $C$ is the constant from \ref{redassump0}.

In particular, for $n_{\delta}\sim\frac{1}{\sqrt{\delta}}$, one has
$
     \Vert u\de-u^{n_{\delta}}\Vert^2  = O(\delta).
$
\end{theo}
\begin{proof}
We plug \eqref{estim-seq-v} in \eqref{estim-norm-ineq} to get an estimate for $\|u\de-u^n\|^2$.
The coefficient of the first term in \eqref{estim-norm-ineq} is non-negative due to \eqref{prod-seq}, thus yielding
\begin{align*} 
    & \Vert u\de-u^n\Vert^2 \leq \tau_n^2\left(\frac{\Delta_0}{\tau_0}  + 2C\delta \sum_{k=1}^{n}\frac{\|v^k-v\de\|}{\tau_{k-1}}\right).
\end{align*}
By using \eqref{estim-seq-v} in this inequality, one gets 
\begin{equation} \label{estim-seq-u}
    \Vert u\de-u^n\Vert^2 \leq \tau_n^2\left(\frac{\Delta_0}{\tau_0}  + 2C\delta \sum_{k=1}^{n}\frac{1}{\tau_{k-1}}\left( 2 C\delta\beta\sum_{i=1}^k\frac{1}{\tau_{i-1}}  + \sqrt{\beta\frac{\Delta_0}{\tau_0}} \right)\right).
\end{equation}
For simplicity, denote the second term of the outer parenthesis by $\mu_n$, which will be estimated in the sequel. 
According to Section 5 from \cite{chambolle2011first},  there exists some $N_0$ such that for all $k\geq N_0,$ $\tau_k\sim \frac{1}{\gamma k}$. We therefore split the sums in $\mu_n$ at the index $N_0$ and take into account this asymptotic behavior for all indices beyond $N_0$, resulting in

\begin{align*}
    &\mu_n = \underbrace{2C\delta \sum_{k=1}^{N_0}\frac{1}{\tau_{k-1}}\left( 2C\delta\beta\sum_{i=1}^k\frac{1}{\tau_{i-1}}
 + \sqrt{\beta\frac{\Delta_0}{\tau_0}}\right)}_{\bar{\mu}_n} + \underbrace{2C\delta\sum_{k=N_0+1}^{n}\frac{1}{\tau_{k-1}}\left( 2 C\delta\beta\sum_{i=1}^k\frac{1}{\tau_{i-1}} + \sqrt{\beta\frac{\Delta_0}{\tau_0}} \right)}_{\tilde{\mu}_n}.
 \end{align*}
 We estimate $\mu_n, \tilde{\mu}_n$ as follows,
 \begin{align*}
     \bar{\mu}_n = & 4C^2\delta^2\beta\sum_{k=1}^{N_0}\left(\frac{1}{\tau_{k-1}}\sum_{i=1}^k\frac{1}{\tau_{i-1}}\right) + 2C\delta \sum_{k=1}^{N_0}\frac{1}{\tau_{k-1}}\sqrt{\beta\frac{\Delta_0}{\tau_0}} \sim \delta^2+\delta,
 \end{align*}

 \begin{align*}
\tilde{\mu}_n & \sim  2C\delta\gamma\sum_{k=N_0}^{n-1}k\left( 2 C\delta\beta\sum_{i=1}^{N_0}\frac{1}{\tau_{i-1}} + 2 C\delta\beta\gamma\sum_{i=N_0}^{k-1}i + \sqrt{\beta\frac{\Delta_0}{\tau_0}} \right)\\
 & \sim C\delta\gamma\left(2C\delta\beta\sum_{i=1}^{N_0}\frac{1}{\tau_{i-1}} + \sqrt{\beta\frac{\Delta_0}{\tau_0}} \right)((n-1)n-(N_0-1)N_0) +  \\
 & 2C^2\delta^2\beta\gamma^2 \sum_{k=N_0}^{n-1}k ((k-1)k-(N_0-1)N_0)\\
 & \sim  C \gamma \sqrt{\beta\frac{\Delta_0}{\tau_0}}\delta n^2 +\frac{1}{2}C^2\beta\gamma^2 \delta^2 n^4.
\end{align*}
Plugging this back in \eqref{estim-seq-u} yields
\begin{align*}
        \Vert u\de-u^n\Vert^2 \sim \frac{1}{\gamma^2n^2}\left(\frac{\Delta_0}{\tau_0}  + C \gamma\sqrt{\beta\frac{\Delta_0}{\tau_0}} \delta n^2 +\frac{1}{2}C^2\beta \gamma^2\delta^2 n^4 \right).
\end{align*}

If one chooses $n_{\delta}$ large enough such that $n_{\delta}\sim\frac{1}{\sqrt{\delta}}$, then one obtains
\begin{align*}
     \Vert u\de-u^{n_{\delta}}\Vert^2 \sim \frac{\delta}{\gamma^2}\left(\frac{\Delta_0}{\tau_0}   + C\gamma\sqrt{\beta \frac{\Delta_0}{\tau_0}} +2C^2\beta\gamma^2 \right), \text{i.e., }
     \Vert u\de-u^{n_{\delta}}\Vert^2  = O(\delta).
\end{align*}
\end{proof}

\begin{remark}
Letting $\delta=0$ in \eqref{estim-seq-u} implies  
$\|u\de-u^n\|^2\sim \frac{1}{n^2}\cdot \frac{\Delta_1}{\tau_1\gamma^2},$ $n\geq N_0,$
which is the rate in \cite{chambolle2011first}.
\end{remark}
\begin{remark}
    For an acceleration of the PDHG algorithm in reflexive Banach spaces, we refer the reader to \cite{darbon2021accelerated}, where the authors study convergence in the case of  exact data. 
    While this approach can be extended to noisy data, the resulting analysis is more technical due to the choice of the associated parameters and lies beyond the scope of this paper.
    \color{black}
\end{remark}

\section{Particular algorithms in Hilbert spaces}\label{sect:concrete-frameworks}

As announced in the Introduction, we specialize our general results to two important choices of data fidelity within the Hilbert space framework. Our main contributions in this setting are error estimates for the non-accelerated Condat-Vũ and accelerated PDHG algorithms applied to Morozov regularization, together with error estimates for the accelerated scheme applied to the equality constrained problem.

\subsection{The equality constraint case}
Consider $H_{f^{\delta}} = i_{\{f^{\delta}\}}$ as in \cite{molinari2024iterative}. Thus, we approach \eqref{intitial-pb-exact} via
\begin{equation*}
    \inf\limits_{u\in \X}\{ R(u)+G(u)\} \text{ such that } \op u=f^{\delta}.
\end{equation*}

The Fenchel conjugate of the data fidelity is $H_{f^{\delta}}^*(z^*) = \langle z^*, f^{\delta}\rangle$ and the subdifferential is $\partial H_{f^{\delta}}^*(z^*) = \{f^{\delta}\}$. In this particular framework, both \ref{A10} and \ref{redassump0} amount to 
\begin{equation*}
    \forall k\in\mathbb{N}, \; \langle f^{\delta} - f, v\de - v^k \rangle \leq \delta \| v\de-v^k\|,
\end{equation*}
which is clearly satisfied due to \eqref{eq:noise_level}.
Moreover, Algorithm \ref{alg:Condat-Vu-Hilbert} amounts to the following.
   \begin{algorithm}[H]
    \caption{The Condat-V\~u algorithm for equality constrained problems}
    \label{Condat-Vu-al-4}
    \begin{algorithmic}[1]
    \State \textbf{Input:} Initial values $u^0\in\operatorname{dom }G', v^0$, step sizes $\tau, \sigma > 0$, data $f^\delta$
    \For{$k = 0, 1, 2, \dots$ }
        \State $u^{k+1} \gets \text{prox}_{\tau R}  \left( u^{k}- \tau (\op^* v^{k} + G'(u^{k})) \right) $
        \State $v^{k+1} \gets  v^{k} + \sigma \op (2u^{k+1} - u^{k})-\sigma f^{\delta} $
    \EndFor
    \State \textbf{Return:} $u^{k+1}, v^{k+1}$
    \end{algorithmic}
    \end{algorithm}
Algorithm \ref{Condat-Vu-al-4} is  very similar to the one from \cite{molinari2024iterative}, the main difference being that the extrapolation step is done for the primal variable $(u^k)$ in the first case (see above) and for the dual $(v^k)$ in the second one - see below:
\begin{gather*}
    \begin{cases}
        u^{k+1} = \text{prox}_{\tau R}  \left( u^{k}- \tau (\op^* (2v^{k+1}-v^{k}) + G'(u^{k})) \right) \\
        v^{k+1} =  v^{k} + \sigma \op u^{k+1}-\sigma f^{\delta}.
    \end{cases}
\end{gather*}
 We omit the full description of the accelerated procedure in this particular instance, pointing out just the corresponding dual update from Algorithm \ref{acc-alg}: $v^{k+1} =  v^{k} + \sigma_k \op \tilde u^{k}-\sigma_k f^{\delta}.$ 

\subsection{The inequality constraint case - Morozov regularization}
Let $H_{f^{\delta}} = i_{C^{\delta}},$ with $C^{\delta} = \{ z : \|f^{\delta}-z\|\leq \delta\}$, which leads to the Morozov regularization problem \eqref{morozov-reg}. We verify below that our assumptions are satisfied in this setting.

The Fenchel conjugate of $H_{f^{\delta}}$ is given by
$
    H_{f^{\delta}}^*(z^*) = \langle z^*, f^{\delta}\rangle + \delta \|z^*\|.
$
Assuming that the source element $v\de$ does not vanish, one can calculate the subdifferential of $H_{f^{\delta}}^*$ at $v\de$ with the help of  Theorem 4.6 in \cite{clason2020introduction}, that is
$
    \partial H_{f^{\delta}}^*(v\de)=f^{\delta}+\delta\cdot\{x^*: \langle x^*,v\de\rangle=\|v\de\|, \|x^*\|=1\}.
$

We now show that \ref{A10} holds.
Let $\zeta^{\delta} \in \partial H_{f^{\delta}}^{*}\left(v\de\right)$ and $\zeta \in \partial H_{f}^{*}\left(v\de\right)$. Then there exists $x^*$ such that  $\langle x^*,v\de\rangle=\|v\de\|$, $\|x^*\|=1$ and $\zeta^{\delta} = f^{\delta} + \delta x^*$. One can evaluate the left-hand side of \eqref{hypo-noisy} to get
$
    \langle f^{\delta} + \delta x^* - f, v\de-v^k \rangle = \langle f^{\delta}  - f, v\de-v^k \rangle + \delta \langle x^* , v\de-v^k \rangle \leq 2\delta \|v\de-v^k\|.
$
Similarly, one can check that assumption \ref{redassump0} is satisfied.

Algorithm \ref{alg:Condat-Vu-Hilbert} takes the following form, given that the dual update has a closed-form expression, cf. \cite{Combettes2005}.
   \begin{algorithm}[H]
    \caption{The Condat-V\~u algorithm for Morozov regularization}
    \label{Condat-Vu-alg-6}
    \begin{algorithmic}[1]
    \State \textbf{Input:} Initial values $u^0\in\operatorname{dom }G', v^0$, step sizes $\tau, \sigma > 0$, data $f^\delta$
    \For{$k = 0, 1, 2, \dots$ }
        \State $u^{k+1} \gets \text{prox}_{\tau R} \left(  u^{k} - \tau (\op^* v^{k} + G'(u^{k})) \right) $
        \State $v^{k+1} \gets \max\left( 1-\frac{\sigma\delta}{\|v^{k}+\sigma \op (2u^{k+1} - u^{k})-\sigma f^{\delta}\|},0\right) (v^{k} + \sigma \op (2u^{k+1} - u^{k})-\sigma f^{\delta}) $
    \EndFor
    \State \textbf{Return:} $u^{k+1}, v^{k+1}$
    \end{algorithmic}
    \end{algorithm} 
We also omit a full presentation of the corresponding specialized accelerated scheme.

\section{The Condat-V\~u algorithm with noisy data in the Banach space setting}\label{sect:Condat-Vu-general}

\subsection{Reflexive Banach spaces}

In the sequel, we establish ergodic convergence rates for the Condat-V\~u algorithm with general data fidelity functions in the presence of noisy data, where both $\X$ and $\Y$ are reflexive Banach spaces. We also
highlight the main differences from the Hilbert space analysis developed in Section~\ref{sect:particular-instances}.

We start by noting that inclusion \eqref{OC1-gen-n} adapted to this setting takes the form
\begin{gather}\label{OC1-gen-n-Banach}
    \Phi(\op u\de)\in\partial H_f^*(v\de),
\end{gather}
where $\Phi:\Y\rightarrow \Y^{**}$ denotes the canonical embedding of $\Y$ in its bidual. The main differences from the Hilbert space setting is that the subdifferential of $H_f^*$ is a subset of $\Y^{**}$ and that the dual variables are in $Y^*$.

Fix $p,q>1$ such that $\frac{1}{p}+\frac{1}{q}=1.$ The function $B$  is defined here via two auxiliary functions $\varphi$ and $\psi$ which determine the non-Euclidian geometries used for the primal and dual variables. 

Recall that a function $h:X\rightarrow\mathbb{R}\cup\{+\infty\}$ is supercoercive if 
$\lim\limits_{\|u\|\rightarrow+\infty}\frac{h(u)}{\|u\|}=+\infty$.
For the definitions of essential smoothness and essential strict convexity, see \cite{bauschke2001essential}.

We assume the following:
\begin{enumerate}[start=10,label={(A\arabic*)}]
    \item $\varphi:\X\rightarrow\mathbb{R}\cup\{+\infty\}$ is proper, lower semicontinuous, essentially smooth, essentially strictly convex, such that $\operatorname{dom}\partial R\subset\operatorname{int}(\operatorname{dom} \varphi)$, and at least one of $R$ and $\varphi$ is supercoercive. Moreover, assume that it is $\mu_1$-strongly convex: \begin{equation*}
        \forall u_1, u_2\in\operatorname{dom }\varphi, \quad D_{\varphi}(u_1, u_2) \ge \frac{\mu_1}{p} \|u_1 - u_2\|^p_\X .
    \end{equation*} \label{A8-gen}
    \item $\psi:\Y^*\rightarrow\mathbb{R}\cup\{+\infty\}$ is proper, lower semicontinuous, essentially smooth, essentially strictly convex, such that $\operatorname{dom}\partial \fidelitynoisy^*\subset\operatorname{int}(\operatorname{dom} \psi)$, and at least one of $\fidelitynoisy^*$ and $\psi$ is supercoercive. Moreover, assume that it is $\mu_2$-strongly convex: 
    \begin{equation*}
        \forall v_1, v_2\in\operatorname{dom }\psi,\quad D_{\psi}(v_1, v_2) \ge \frac{\mu_2}{q} \|v_1 - v_2\|^q_{\Y^*}.
    \end{equation*}\label{A9-gen}
\end{enumerate}

With these clarifications, the function $B:\X\times \Y^*\rightarrow\mathbb{R}\cup\{+\infty\}$ is chosen as
\begin{equation}\label{def-B-gen}
B(u,v) = \frac{1}{\tau}\varphi(u) + \frac{1}{\sigma}\psi(v) - \langle v, \op u \rangle - G(u),
\end{equation}
where $\tau, \sigma>0$. Note that $\langle v,\op u\rangle = \langle \Phi(\op u),v\rangle.$

The  Condat-V\~u algorithm can again be derived from the fixed-point iteration
\begin{gather}\label{fpi-gen-Banach}
    \begin{pmatrix}
        0 \\ 0
    \end{pmatrix} \in
    \begin{pmatrix}
        \partial R\left(u^{k+1}\right)+G'\left(u^{k+1}\right)+\op^{*} v^{k+1} \\ \partial H_{f^{\delta}}^{*}\left(v^{k+1}\right)-\Phi(\op  u^{k+1})
    \end{pmatrix}+\nabla B\left(u^{k+1}, v^{k+1}\right)-\nabla B\left(u^{k}, v^{k}\right),
\end{gather}
 and  takes the following form.

\begin{algorithm}[H]
\caption{Generalized Condat-V\~u algorithm}
\label{Condat-Vu-alg}
\begin{algorithmic}[1]
\State \textbf{Input:} Initial values $u^0\in\operatorname{dom }G'\cap\operatorname{dom}\partial R\color{black}, v^0\in\operatorname{dom}\partial \fidelitynoisy^*\color{black}$, step sizes $\tau, \sigma > 0$, data $f^\delta$
\For{$k = 0, 1, 2, \dots$ }
    \State $u^{k+1} \gets \arg \min\limits_{u\in \X} G(u^{k}) + \langle G'(u^{k}), u - u^{k} \rangle +R(u)+\langle \op u, v^{k}\rangle+ \frac{1}{\tau} D_\varphi(u, u^{k}) $
    \State $v^{k+1} \gets \arg \min\limits_{v\in \Y^*} H^*_{f^{\delta}}(v) - \langle v, \op (2u^{k+1} - u^{k}) \rangle + \frac{1}{\sigma} D_\psi(v,v^{k})$
\EndFor
\State \textbf{Return:} $u^{k+1}, v^{k+1}$
\end{algorithmic}
\end{algorithm}
Note that this procedure is an adaptation of the one from \cite{chambolle2016ergodic} for noisy data.

Assumptions \ref{A2} - \ref{A4}, \ref{A8-gen} and \ref{A9-gen} imply that Algorithm \ref{Condat-Vu-alg} is well-defined - for more details see Appendix~A, Fact~A.8 from \cite{darbon2021accelerated}.

The convergence of the algorithm is discussed in \cite{chambolle2016ergodic}. One needs  the assumption  $\tau\sigma\|\op \|^2<1$ and existence of  a saddle point for \eqref{saddle-point-problem}  - the latter being  guaranteed by \ref{source-condition}. The authors remark that both the sequence of iterates $(u^n,v^n)$, as well as the ergodic averages $(\bar{u}^n,\bar{v}^n)$, are bounded. Additionally, any weak cluster point of the ergodic averages is a saddle point of problem \eqref{saddle-point-problem}.

We show below that $B$ is strictly convex, despite the fact that just convexity would suffice in the subsequent analysis.
Note that one does not need  source condition \ref{source-condition} in the next lemma.

\begin{lemma}\label{lem-str-conv}
    Let $\tau, \sigma>0$ be fixed. Let \ref{A4}, \ref{A8-gen}, \ref{A9-gen} hold and assume the inequality $\tau(\sigma^{p-1}\|\op \|^p+\mu_2^{p-1}L)<\mu_1\mu_2^{p-1}.$
    Then $B$ defined by \eqref{def-B-gen} is a strictly convex function.
\end{lemma}
\begin{proof}
  We will show that $\nabla B$ is a strictly monotone operator. Fix $w = (u,v), \tilde{w} = (\tilde{u}, \tilde{v})$ such that $w\neq \tilde{w}$. Then one has
    \begin{gather}
        \langle \nabla B(w)-\nabla B(\tilde{w}),\begin{pmatrix}
                                                       u-\tilde{u}\\v-\tilde{v} 
                                                    \end{pmatrix}\rangle =
        \frac{1}{\tau}D_{\varphi}^s(u,\tilde{u}) + \frac{1}{\sigma}D_{\psi}^s(v,\tilde{v})-D_{G}^s(u,\tilde{u})-2\langle \op (u-\tilde{u}), v-\tilde{v}\rangle. \label{eq}
    \end{gather}
    Since $G$ is $L$-smooth, it follows that
    $
        D_G(u,\tilde{u})\leq\frac{L}{p}\Vert u-\tilde{u}\Vert^p_\X.
    $
    Changing the roles of $u$ and $\tilde{u}$, and summing the inequalities yield
    \begin{equation}\label{ineq3}
        D^s_G(u,\tilde{u})\leq \frac{2L}{p}\Vert u-\tilde{u}\Vert^p_\X.
    \end{equation} 
    From \eqref{eq}, \eqref{ineq3} and the strong convexity of $\varphi$ and $\psi$, one gets
    \begin{gather}
        \langle \nabla B(w)-\nabla B(\tilde{w}),\begin{pmatrix}
                                                       u\\v 
                                                    \end{pmatrix}-
                                                    \begin{pmatrix}
                                                       \tilde{u}\\ \tilde{v} 
                                                    \end{pmatrix}\rangle\geq
        \left(\frac{2\mu_1}{\tau}-2L\right)\frac{\Vert u-\tilde{u}\Vert^p_\X}{p} + \frac{2\mu_2}{\sigma} \frac{\Vert v-\tilde{v}\Vert^q_{\Y^*}}{q} - 2\langle u-\tilde{u}, \op ^*(v-\tilde{v})\rangle. \label{ineq4}
    \end{gather}
    From the inequalities of Cauchy and Young, respectively, we can estimate the last term of the sum,
    \begin{gather*}
        \langle u-\tilde{u}, \op ^*(v-\tilde{v})\rangle\leq \|\op \| \cdot \Vert u-\tilde{u}\Vert_\X \cdot \Vert v-\tilde{v}\Vert_{\Y^*}\leq \bar{\epsilon}\|\op \|^p\frac{\Vert u-\tilde{u}\Vert^p_\X}{p}+\frac{\Vert v-\tilde{v}\Vert^q_{\Y^*}}{q\bar{\epsilon}^{q/p}},
    \end{gather*}
    for some $\bar{\epsilon}>0.$
    Plugging this back in \eqref{ineq4}, it follows
    \begin{equation*}
        \langle \nabla B(w)-\nabla B(\tilde{w}),w-\tilde{w}\rangle> \left(\frac{2\mu_1}{\tau}-2L-2\bar{\epsilon}\|\op \|^p\right) \frac{\Vert u-\tilde{u}\Vert^p_\X}{p} + \left( \frac{2\mu_2}{\sigma}-\frac{2}{\bar{\epsilon}^{q/p}} \right) \frac{\Vert v-\tilde{v}\Vert^q_{\Y^*}}{q}.
    \end{equation*} 
    One can choose $\bar{\epsilon}>0$ such that the terms in both brackets are positive. Indeed, due to our assumptions, one has $\frac{\sigma^{p-1}}{\mu_2^{p-1}} < \frac{1}{\|\op \|^p}\left(\frac{\mu_1}{\tau}-L\right) $, thus one can pick an $\bar{\epsilon}$ such that
$
            \bar{\epsilon} \in \left(\frac{\sigma^{p-1}}{\mu_2^{p-1}} , \frac{1}{\|\op \|^p}\left(\frac{\mu_1}{\tau}-L\right)\right).
$
    Then, one gets
    $
        \langle \nabla B(w)-\nabla B(\tilde{w}),w-\tilde{w}\rangle > 0.
    $
\end{proof}

We state below the main result concerning error estimates for the ergodic averages $
    \overline{u}^n := \frac{1}{n}\sum_{k = 1}^n u^k, \quad \overline{v}^n := \frac{1}{n}\sum_{k = 1}^n v^k
$
 of Algorithm \ref{Condat-Vu-alg}.

\begin{theo}\label{main-th-gen}
    Let  $\tau, \sigma>0$ be fixed. Assume that \ref{A1} - \ref{A10},  \ref{A8-gen} and \ref{A9-gen}  hold, and that $\tau(\sigma^{p-1}\|\op \|^p+\mu_2^{p-1}L)<\mu_1\mu_2^{p-1}.$
    Then there exists $\nu > 0$ and $\bar{C}>0$ such that, for any $n\in\mathbb{N}$, one has
\begin{gather}\label{estim:non-acc-gen}
  D_{R+G}(\overline{u}^n, u^\dagger) + D_{H_{f^{\delta}}^\star}\left(\overline{v}^n, v^\dagger\right)  \leq \frac{1}{n} D_B(w\de,w^{0}) + \frac{\bar{C}^p q^{p/q}}{\nu^{p/q}} \delta^p n^{p-1}+ C \left(\frac{pq}{\nu}D_B(w\de,w^0)\right)^{1/q} \delta .
\end{gather}
\end{theo}
\begin{proof}
Fix $n\geq1$ and $k$ with $0\leq k\leq n-1$. We follow the same steps as in the proof of Theorem \ref{main-theo-part} and only point out the main differences that appear in this setting. Equation \eqref{fixed-point-iteration-part} becomes:
\begin{gather}\label{fixed-point-Banach}
    \begin{pmatrix}
        0 \\ 0
    \end{pmatrix} \in
    \begin{pmatrix}
        \partial R\left(u^{k+1}\right)-\partial R\left(u\de\right)+G'\left(u^{k+1}\right)-G'\left(u\de\right)+\op^{*} (v^{k+1}-v\de) \\ \partial H_{f^{\delta}}^{*}\left(v^{k+1}\right)-\partial H_f^{*}\left(v\de\right)-\Phi(\op u^{k+1})+\Phi(\op u\de)
    \end{pmatrix}+\nabla B\left(w^{k+1}\right)-\nabla B\left(w^{k}\right).
\end{gather}

We take the dual product with $w^{k+1} - w^\dagger$, employ the three-point identity for $B$ and get the same inequality as in \eqref{ineq-Bregman-n},
\begin{equation} \label{ineq-Bregman-n0}
    \sum\limits_{k=1}^n{\langle  \partial H_{f^{\delta}}^{*}\left(v\de\right)-\partial H_f^{*}\left(v\de\right),v\de-v^k\rangle} +D_B(w\de,w^0)\geq D_B(w\de,w^n).
\end{equation}
The right-hand side can be rewritten as follows,
\begin{gather*}
    D_B(w\de,w^n)=
    \frac{1}{\tau} D_{\varphi}(u\de, u^n) + \frac{1}{\sigma} D_{\psi}(v\de, v^n) - D_G(u\de, u^n) + \langle \op (u\de - u^n), v^n - v\de \rangle.
\end{gather*}
    Following the same steps as in the proof of Lemma \ref{lem-str-conv}, one gets
    \begin{gather}\label{inq}
        D_B(w\de,w^n)\geq \left(\frac{\mu_1}{\tau}-L - \bar{\epsilon}\Vert \op \Vert^p\right)\frac{\Vert u\de-u^n\Vert^p_\X}{p} + \left( \frac{\mu_2}{\sigma}-\frac{1}{\bar{\epsilon}^{q/p}}\right)\frac{\Vert v\de-v^n\Vert^q_{\Y^*}}{q}.
    \end{gather}
     Denote by $\nu = \frac{\mu_2}{\sigma} - \frac{1}{\bar{\epsilon}^{q/p}}$.
    Both coefficients (in brackets) of the right-hand side of \eqref{inq-appendix} are positive, which implies 
    $
            D_B(w\de,w^n)\geq\nu\frac{\Vert v\de-v^n\Vert^q_{\Y^*}}{q}.
    $

    This, together with \eqref{ineq-Bregman-n0}, gives 
    $
        \Vert v\de-v^n\Vert^q_{\Y^*} \leq \frac{q}{\nu}\sum\limits_{k=1}^n{\langle  \partial H_{f^{\delta}}^{*}\left(v\de\right)-\partial H_f^{*}\left(v\de\right),v\de-v^k\rangle} +\frac{q}{\nu} D_B(w\de,w^0).
    $
    Consequently, by employing \eqref{hypo-noisy}, one gets
    \begin{equation}\label{estim-dual-CV-appendix}
        \Vert v\de-v^n\Vert^q_{\Y^*} \leq \frac{qC\delta}{\nu}\sum\limits_{k=1}^n{\Vert v\de-v^k\Vert_{\Y^*}} +\frac{q}{\nu} D_B(w\de,w^0).
    \end{equation}

Now the idea is to use Theorem \ref{lemma-Schmidt-gen}.
For this, we will make the choices: $a_n := \Vert v\de- v^n\Vert_{\Y^*}\geq0,$ $S_n := \frac{q}{\nu}D_B(w\de,w^0) \text{ non-decreasing}, $ $\lambda_n := \frac{qC{\delta}}{\nu} \geq 0,$
for any $n\geq 1$.
Then, for every $k$, 
\begin{gather} 
    \Vert v\de- v^{k+1}\Vert_{\Y^*}\leq  \left(\frac{pq}{\nu}D_B(w\de,w^0)\right)^{1/q} + \left( \left(\frac{(k+1)qC{\delta}}{\nu}\right)^p\right)^{1/q}, \label{estim-dualvar0}
\end{gather}
where we have also employed the property ${(a+b)}^{1/q}\leq{a}^{1/q}+{b}^{1/q},$ for $a, b$ non-negative and $q>1$.

 Using inequalities \eqref{hypo-noisy} and \eqref{estim-dualvar0} in \eqref{main-eq-n}, one finds
\begin{gather}
    D_R^{s}(u^{k+1}, u^\dagger)+D_G^{s}(u^{k+1}, u^\dagger) + D_{H_{f^{\delta}}^\star}^{s}(v^{k+1}, v^\dagger) \leq \nonumber \\ C \delta \left(\left(\frac{pq}{\nu}D_B(w\de,w^0)\right)^{1/q} + \left( \frac{(k+1)qC{\delta}}{\nu}\right)^{p/q}\right) - D_B(w\de,w^{k+1})-D_B(w^{k+1},w^{k})+D_B(w\de,w^{k}). 
\end{gather}
Summing up from $0$ to $n-1$ and using the sum rule for the subdifferential implies
\begin{gather*}
    \sum\limits_{k=1}^{n}{\left(D_{R+G}^{s}(u^k, u^\dagger) + D_{H_{f^{\delta}}^\star}^{s}(v^k, v^\dagger)\right)} \leq
    D_B(w\de,w^{0}) +  \frac{C^p\delta^p q^{p/q}}{\nu^{p/q}}\sum_{k=1}^n { k^{p/q}} + C\delta n \left(\frac{pq}{\nu}D_B(w\de,w^0)\right)^{1/q} .
\end{gather*}
It remains to estimate $\sum_{k=1}^n { k^{p/q}}$, which is the same as $\sum_{k=1}^n { k^{p-1}}$. According to the generalized Faulhaber formula (Corollary 15, \cite{schumacher2022generalization}), this term is of order $O(n^p)$.
With this, one finds
\begin{gather*}
 D_{R+G}(\overline{u}^n, u^\dagger) + D_{H_{f^{\delta}}^\star}\left(\overline{v}^n, v^\dagger\right)  \leq \frac{1}{n} D_B(w\de,w^{0}) + \frac{\bar{C}^p q^{p/q}}{\nu^{p/q}} \delta^p n^{p-1}+ C \left(\frac{pq}{\nu}D_B(w\de,w^0)\right)^{1/q} \delta,
\end{gather*}
for $\overline{u}^n = (\sum_{k = 1}^n u^k)/n$ and $\overline{v}^n = (\sum_{k = 1}^n v^k)/n$. 
\end{proof}

\subsection{Non-reflexive Banach spaces: Entropic PDHG} \label{ex:entropic_mirror_example} 
    In what follows, we focus on the recovery of a nonnegative solution $\truesol$ of $Ku=f$, that belongs to the probability simplex $\Delta = \{ u \in L^1(\Omega) : \int\limits_{\Omega} u \; dt = 1, u\geq0 \text{ a.e.}\}$.  Here we consider  $\op :L^1(\Omega)\rightarrow L^2(\Gamma)$, where  $\Omega, \Gamma\subset \mathbb{R}^2$ are two measurable sets. We choose the function $\varphi$ to be the negative Boltzmann-Shannon entropy,
\begin{gather*}     
    \varphi(u) = \begin{cases}
                    \int\limits_{\Omega} u\operatorname{ln} u \; dt, &\text{ if } u\geq0 \text{ a.e. and } u\operatorname{ln} u  \in L^1(\Omega),\\
                    +\infty, &\text{ otherwise,}
                \end{cases}
    \end{gather*} 
    and the function $B$ from \eqref{def-B-gen} to have the form 
    \begin{equation}
        B(u,v) = \frac{1}{\tau}\varphi(u) + \frac{1}{2\sigma}\|v\|^2_2 - \langle v, \op u \rangle.
    \end{equation}
    The natural space for $\varphi$ is $L^1(\Omega)$, whose positive cone has  empty interior. Therefore, the derivative of $\varphi$ (and consequently, that of  $B$) used in the previous proofs is no longer available. However, we adapt the arguments to this setting by employing directional derivatives of $\varphi$.

\subsubsection{Well-definedness}
    \begin{prop}\label{prop:entropic_mirror_example} 
    Assume that the penalty $R$ is the indicator function over the probability simplex, $G=0,$ and $\psi(v) = \frac{1}{2}\|v\|_2^2$. 
    Then Algorithm \ref{Condat-Vu-alg} is well-defined.
\end{prop}
\begin{proof}
    The minimization problems defining Algorithm \ref{Condat-Vu-alg} can be expressed as follows, 
   \begin{gather}\label{minimization-pbs-KL}
    \begin{cases}
        u^{k+1} = \arg \min\limits_{u\in L^1} i_{\Delta}(u) + \frac{1}{\tau}KL(u, u^k) + \langle Ku, v^k\rangle\\
       v^{k+1} = \arg \min\limits_{v\in L^2} H^*_{f^{\delta}}(v) + \frac{1}{\sigma} \frac{\|v-v^k\|_2^2}{2} + \langle v, K(2u^{k+1}-u^k)\rangle.
    \end{cases}
    \end{gather}
    Notice that $H_{f^\delta}^*$ is convex and weakly$^*$-lower semicontinuous, being a Fenchel conjugate. Together with the quadratic and linear terms, this implies that the objective functional of the second problem is weakly$^*$-lower semicontinuous and coercive. Hence, the corresponding sublevel sets are bounded and weakly$^*$-closed, and are therefore weakly$^*$-compact, due to Banach-Alaoglu theorem.


According to \cite[Proposition 2.3]{burger2020entropic}, the first problem of \eqref{minimization-pbs-KL} has a unique solution in $\operatorname{dom}(\varphi+i_{\Delta})$ given by $u = \frac{1}{\int_{\Omega}u^ke^{-\tau K^*v^k}\; dt}u^ke^{-\tau K^*v^k}$, which satisfies $u\in \operatorname{dom}\partial\varphi$. 
Specifically,  the updates of the algorithm become
\begin{gather*}
\begin{cases}
     u^{k+1} = \frac{u^k\exp(-\tau \op^*v^k)}{\int_{\Omega} u^k\exp(-\tau \op ^*v^k)\; dt} \in\operatorname{dom}\partial\varphi,\\
    v^{k+1} = \text{prox}_{\sigma H_{f^{\delta}}^*} \left(   v^{k} + \sigma \op  (2u^{k+1} - u^{k}) \right), 
\end{cases}
\text{ where } u^0\in\operatorname{dom }\partial\varphi.
\end{gather*}
Then, from \cite[Lemma 2.2 (iv)]{burger2020entropic}, one also finds $\partial KL(\cdot, u^k)(u^{k+1}) = \{\operatorname{ln}u^{k+1} - \operatorname{ln}u^{k}\}$.
\end{proof}


\subsubsection{Convexity of the function $B$ and error estimates}

\begin{prop}
    Suppose that the assumptions of Proposition \ref{prop:entropic_mirror_example} hold and $\tau\sigma\|K\|^2<1$. Then, the function $B$ is convex. 
    If, additionally,  source condition \ref{source-condition} is satisfied, then there exist $\nu>0$ and $\bar{C}>0$ such that for any $n\in\mathbb{N},$ one has
    \begin{equation}\label{ineq:conclu-KL-th}
        D_R(\bar{u}^n, u\de) + D_{H_{\noisydata}^*} (\bar{v}^n, v\de) \leq \frac{1}{n} D_B(w\de, w^0) + \frac{2\bar{C}^2}{\nu}\delta^2n+ 2C\sqrt{\frac{D_B(w\de, w^0)}{\nu}}.
    \end{equation}
\end{prop}
\begin{proof}
We begin by proving the convexity of $B$ under the assumption $\tau\sigma\|\op \|^2<1$.
Note that one can rewrite 
$
    \frac{1}{2\sigma}\|v\|_2^2 - \langle v, \op u \rangle = \left\| \frac{1}{\sqrt{2\sigma}}v - \frac{\sqrt{2\sigma}}{2}\op u \right\|_2^2 - \frac{\sigma}{2}\|\op u\|^2_2,
$
hence 
\begin{equation*}
B(u,v) = \left\| \frac{1}{\sqrt{2\sigma}}v - \frac{\sqrt{2\sigma}}{2}\op u \right\|_2^2 + \underbrace{\frac{1}{\tau}\varphi(u)  - \frac{\sigma}{2}\|\op u\|^2_2}_{F(u)} .
\end{equation*}
The first term is jointly convex as the squared norm of an affine mapping; thus, it suffices to prove the convexity of $F$.

For this, consider $u_0,u_1\in \Delta$ and define $u_t = (1-t)u_0 + tu_1, h = u_1-u_0.$
One can easily check the following identity,
\begin{equation*}
    (1-t)\varphi(u_0) +  t\varphi(u_1) - \varphi (u_t) = (1-t) D_{\varphi}(u_0, u_t) + t D_{\varphi}(u_1, u_t),
\end{equation*}
which, together with Pinsker's inequality \cite[Lemma 2.1 (v)]{burger2020entropic}, $\forall x,y\in \Delta, D_{\varphi}(x,y)\geq \frac{1}{2}\|x-y\|_1^2$, implies
\begin{equation}\label{ineq-entr}
    \varphi(u_t) \leq (1-t)\varphi(u_0) + t\varphi(u_1) - \frac{t(1-t)}{2}\|h\|_1^2. 
\end{equation}
On the other hand, $\op u_t = (1-t) \op u_0 + t\op u_1 = \op u_0 + t\op h,$ and the following equalities hold
\begin{equation*}
    \|\op u_t\|_2^2 = \|\op u_0\|^2_2 +t^2\| \op h\|_2^2+2t \langle \op u_0, \op h\rangle,
\end{equation*}
\begin{equation*}
    \|\op u_1\|_2^2 = \|\op u_0\|^2_2 +\| \op h\|_2^2+2 \langle \op u_0, \op h\rangle.
\end{equation*}
By substituting the term $2\langle \op u_0, \op h\rangle$ from the second equality into the first one, we find
\begin{equation}\label{id-norm}
\|\op u_t\|_2^2 = (1-t)\|\op u_0\|^2_2 +t\| \op u_1\|_2^2 - t(1-t) \|\op h\|^2_2.
\end{equation}
From \eqref{ineq-entr} and \eqref{id-norm} it follows $F(u_t)\leq   (1-t) F(u_0) + tF(u_1) + t(1-t)\left( \frac{\sigma}{2}\|\op h\|_2^2 - \frac{1}{2\tau} \|h\|_1^2\right).$
Since the operator $\op $ is bounded, the previous inequality becomes 
\begin{equation}\label{ex:KL-proof-ineq}
    F(u_t)\leq   (1-t) F(u_0) + tF(u_1) + t(1-t)\left( \frac{\sigma}{2}\|\op \|^2 - \frac{1}{2\tau} \right)\|h\|_1^2.
\end{equation}
By employing the assumption $\sigma\tau\|\op \|^2<1$, the convexity follows.

As for the error estimates, the results follow similarly to those in Theorem \ref{main-th-gen}, by showing directly that an equality of type \eqref{main-eq-n} holds, namely
\begin{gather*}
D_R(u^{k+1}, u^\dagger) + D_{H_{f^{\delta}}^\star}^{s}(v^{k+1}, v^\dagger) = \\
    \left\langle  \partial H_{f^{\delta}}^{*}\left(v\de\right)-\partial H_f^{*}\left(v\de\right),v\de-v^{k+1}\right\rangle  - D_B(w\de,w^{k+1})-D_B(w^{k+1},w^{k})+D_B(w\de,w^{k}),
\end{gather*}
where $D_B$ is defined by means of the directional derivative. We omit the details because of the similarity.
\end{proof}
\begin{remark}
    Note that $D_R(\bar{u}^n, u\de)$ from \eqref{ineq:conclu-KL-th} reduces in this case to $\langle \op^*v\de, \bar{u}^n - \truesol \rangle$.
\end{remark}

\section{Numerical results} \label{sect:numerics}
This section illustrates the performance of the proposed algorithms when applied to a sparsity problem and a hyperspectral abundance map
recovery.

Before presenting the numerical results, we briefly review several relevant instances of the Condat-V\~u algorithm studied in  \cite{chambolle2011first, chambolle2016ergodic, molinari2024iterative}.


\begin{table}[htbp]
    \centering
    \begin{tabular}{l|c|c|c|c|c|c|c}
         & $\X$ & $\Y$ & $H_{f^{\delta}}$ & $R$ & $G$ & $\varphi$ & $\psi$ \\ 
        \hline 
        \begin{tabular}{@{}l@{}} \scriptsize\textsc{Chambolle,} \\ \scriptsize\textsc{Pock (2011)} \end{tabular} 
        & $\mathcal{\mathbb{R}}^{m\times n }$ & $\X\times \X$ & $\frac{\lambda}{2}\|u-f^{\delta}\|^2_2$ & $\|\nabla u\|_1$ & $0$ & $\frac{1}{2}\|u\|^2$ & $\frac{1}{2}\|v\|^2$ \\ \cline{2-8}
        & $\mathcal{\mathbb{R}}^{m\times n}$ & $\X\times \X$ & $\lambda\|u-f^{\delta}\|_1$ & $\|\nabla u\|_1$ & $0$ & $\frac{1}{2}\|u\|^2$ & $\frac{1}{2}\|v\|^2$ \\ \cline{2-8}
        & $\mathcal{\mathbb{R}}^{m\times n}$ & $\X\times \X$ & $ \frac{\lambda}{2}\|u-f^{\delta}\|^2_2$ & $\|\nabla u\|_{\alpha}$ \scriptsize(Huber) & $0$ & $\frac{1}{2}\|u\|^2$ & $\frac{1}{2}\|v\|^2$ \\ \cline{2-8}
        & $\mathcal{\mathbb{R}}^{m\times n}$ & $\X\times \X$ & $\frac{\lambda}{2}\|\op u-f^{\delta}\|^2_2 $ & $\|\nabla u\|_1$ & $0$ & $\frac{1}{2}\|u\|^2$ & $\frac{1}{2}\|v\|^2$ \\ 
        \hline 
       \scriptsize\textsc{Chambolle,} & $\mathbb{R}^l$ & $\mathbb{R}^k$ & $H(\op u) = \max\limits_i(\op u)_i$ & $i_{\Delta_l}$ & $0$ & $\frac{1}{2}\|u\|^2$ & $\frac{1}{2}\|v\|^2$ \\ \cline{2-8}
       \scriptsize\textsc{Pock (2015)} & $\mathbb{R}^l$ & $\mathbb{R}^k$ & $H(\op u) = \max\limits_i(\op u)_i$ & $i_{\Delta_l}$ & $0$ & $u\log u$ & $v\log v$ \\ \cline{2-8}
       \scriptsize$f^{\delta}=f$ & $\mathbb{R}^l$ & $\mathbb{R}^k$ & $H_f(\op u) = \frac{1}{2}\|\op u-f\|^2$ & $i_{\Delta_l}$ & $0$ & $\frac{1}{2}\|u\|^2$ & $\frac{1}{2}\|v\|^2$ \\ \cline{2-8}
        & $\mathbb{R}^l$ & $\mathbb{R}^k$ & $H_f(\op u) = \frac{1}{2}\|\op u-f\|^2$ & $i_{\Delta_l}$ & $0$ & $u \log u$ & $\frac{1}{2}\|v\|^2$  \\ \cline{2-8}
        & $\mathbb{R}^l$ & $\mathbb{R}^k$ & $H_f(\op u) = \frac{1}{2}\|\op u-f\|^2$ & $\lambda_1\|u\|_1$ & $\lambda_2 \|u\|^2$ & $\frac{1}{2}\|u\|^2$ & $\frac{1}{2}\|v\|^2$ \\
        \hline 
        \begin{tabular}{@{}l@{}} \scriptsize\textsc{Molinari,} \\ \scriptsize\textsc {et al.(2022)} \\ \scriptsize$\|f^\delta-f\|_2\le\delta$ \end{tabular} 
        & $l^2$ & $\mathcal{\Y}$ & $H_{f^{\delta}} = i_{\{f^{\delta}\}}$ & $\|u\|_1$ & $0$ & $\frac{1}{2}\|u\|^2$ & $\frac{1}{2}\|v\|^2$ \\ \cline{2-8}
        & $\mathbb{R}^{d\times d}$ & $\mathbb{R}^{d\times d}$ & $H_{f^{\delta}} = i_{\{f^{\delta}\}}$ & $\|\cdot\|_*$\scriptsize (nuclear)& $0$ & $\frac{1}{2}\|u\|^2$ & $\frac{1}{2}\|v\|^2$ \\
        \hline
    \end{tabular}
    \vspace{0.4cm}
    \caption{Summary of several imaging and optimization settings approached via Condat–V\~u algorithms in the literature}
    \label{tab:pdhg_comparison}
\end{table}
\color{black}

\subsection{The sparsity problem}
As in Section \ref{sect:concrete-frameworks}, we consider  Algorithms \ref{Condat-Vu-al-4} and \ref{Condat-Vu-alg-6} for $R=\|\cdot\|_1$ and $G=0$.
The choice of the forward operator $\op $ is inspired by the examples presented in Section 8 of \cite{molinari2024iterative}. Specifically, let $\op \in\mathbb{R}^{100}\times\mathbb{R}^{100}$ be a matrix with Gaussian entries and Toeplitz correlation structure, with the correlation parameter $\rho = 0.5.$ The exact solution $u^{\dagger}$ is chosen to be sparse, with fifteen randomly selected nonzero entries and scaled to satisfy $\|u\de\|_2=20$. The exact data are generated as $f=\op u^{\dagger}$, and noisy observations are constructed by perturbing the exact data with Gaussian noise with a given noise level $\delta$. Therefore, fixing the exact solution norm eliminates the variations in reconstruction quality caused by changes in signal magnitude.

\subsubsection{The non-accelerated PDHG algorithm}

\noindent \textit{The equality constraint case}

\begin{equation*}
    \inf\limits_{u\in \mathbb{R}^{100}}\{ \|u\|_1\} \text{ such that } \op u=f^{\delta}.
\end{equation*}
We approach this using the PDHG method with step sizes $\tau = \sigma = \frac{0.9}{\|\op \|_2}$, where $\|\op \|_2$ denotes the spectral norm of the matrix $\op $. Figure \ref{fig:eq-different-noise-levels} displays the semi-convergence behavior of PDHG for different noise levels $\delta\in[1,15]$ (see also \cite{molinari2024iterative}). 

\noindent \textit{The inequality constraint case - Morozov regularization}

Next, assume that $H_{f^{\delta}} = i_{\{C^{\delta}\}}$ with $C^{\delta} = \{ z : \|f^{\delta}-z\|_2\leq \delta\}$. The resulting problem is:
\begin{equation}
    \inf\limits_{u\in\mathbb{R}^{100}}\{ \|u\|_1\} \text{ such that } \|\op u-f^{\delta}\|_2\leq\delta,
\end{equation}
for which we employ again the PDHG method with the same step-size parameters. Figure \ref{fig:Morozov-different-noise-levels} displays the convergence behavior of this approach for several $\delta\in[1,15].$ 

\begin{figure}[H]
    \centering
    \begin{subfigure}[t]{0.49\textwidth}
        \centering
    \includegraphics[width=\textwidth, height=0.2\textheight, keepaspectratio=false, trim=20 20 60 70, clip]{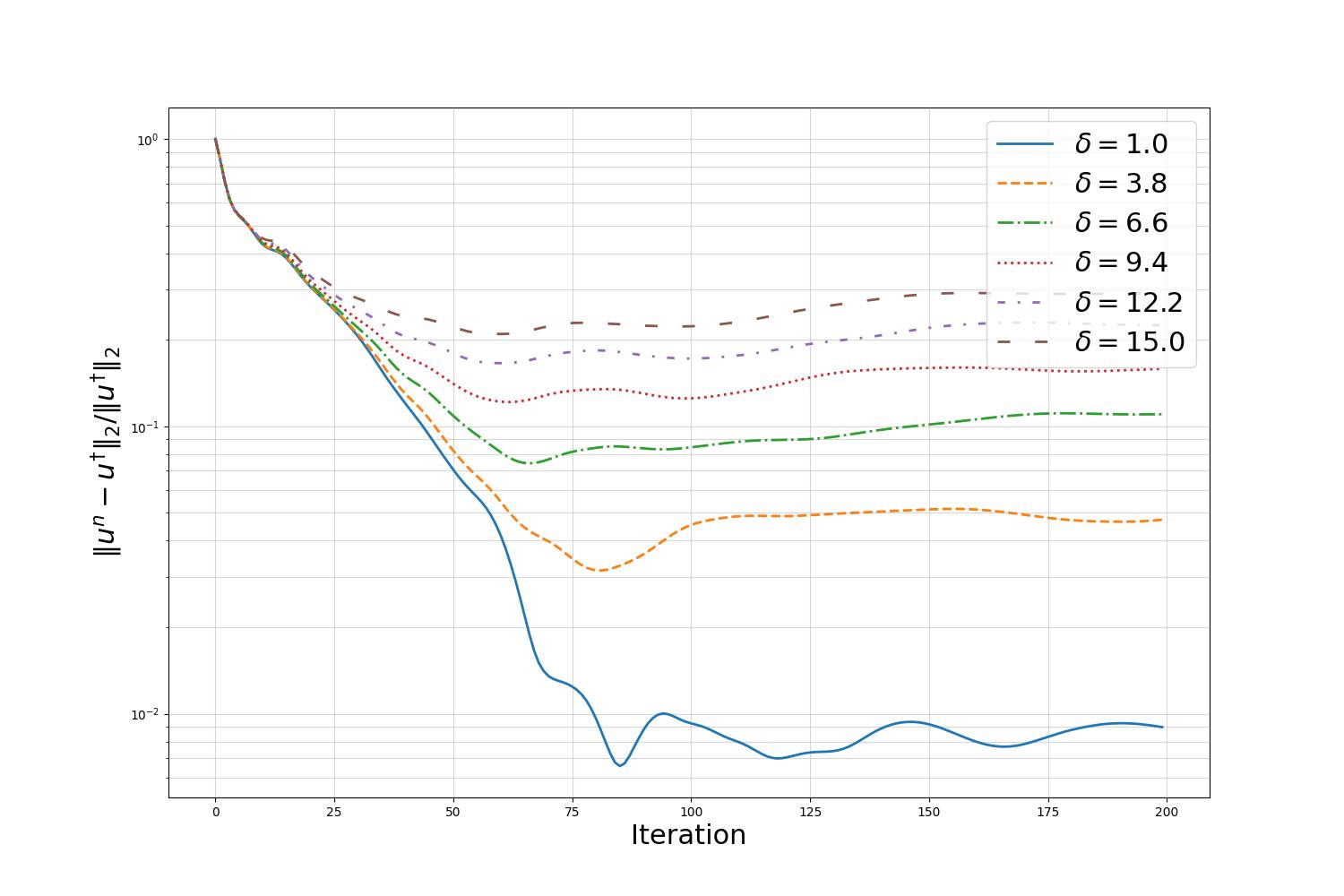}
        \caption{Equality constraint}
        \label{fig:eq-different-noise-levels}
    \end{subfigure}
    \hfill 
    \begin{subfigure}[t]{0.49\textwidth}
        \centering
        \includegraphics[width=\textwidth, height=0.2\textheight, keepaspectratio=false, trim=20 20 60 70, clip]{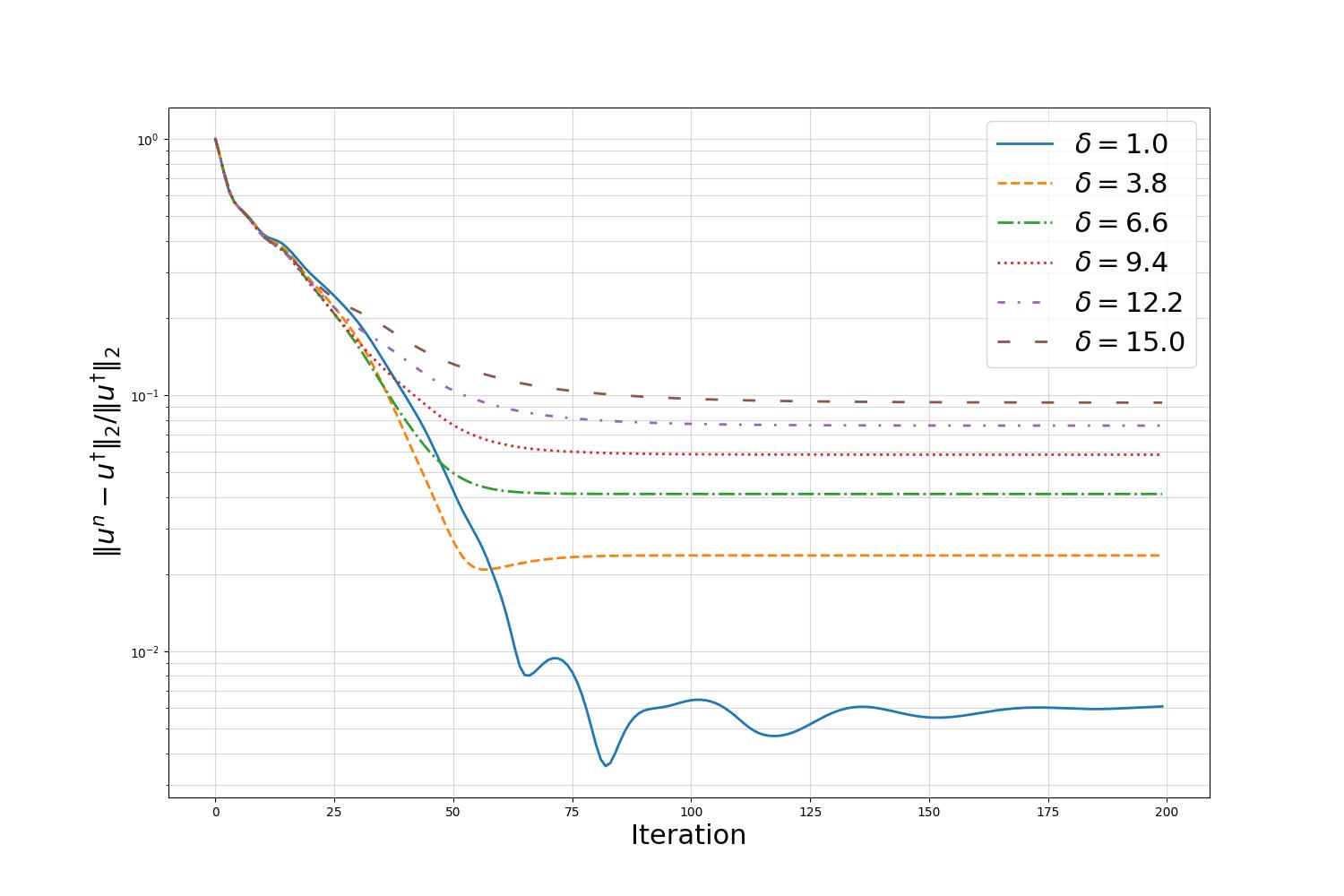}
        \caption{Morozov regularization}
        \label{fig:Morozov-different-noise-levels}
    \end{subfigure}
    \caption{Semi-convergence of the PDHG algorithm for different noise levels.}
    \label{fig:pdhg-semi-convergence}
\end{figure}

\noindent\textit{Equality versus  inequality constraint}

This paragraph compares the performance of the PDHG schemes applied to the equality constrained formulation and to Morozov regularization.
 Table \ref{table:eq-morozov-comparison} highlights for both situations the iteration index $n_{\text{min}}$ at which the smallest reconstruction error $D_R(\bar{u}^n, u\de)$ is achieved, and the corresponding value $D_R(\bar{u}^{n_{\text{min}}}, u\de)$ for different noise levels $\delta\in[1.6, 6]$. Morozov regularization attains lower errors across all tested noise levels.
This confirms that early stopping is the sole regularization mechanism for the equality constrained problem, whereas in the Morozov formulation it provides additional stability to that already induced by the data-fitting inequality.

\begin{table}[H]
\centering
\begin{tabular}{ccccc}
\cline{2-5}
\multicolumn{1}{l|}{}                                  & \multicolumn{2}{c|}{\cellcolor[HTML]{DAD7D7}Equality}                                                              & \multicolumn{2}{c|}{\cellcolor[HTML]{DAD7D7}Morozov}                                                               \\ \hline
\rowcolor[HTML]{EBEBEB} 
\multicolumn{1}{|c|}{\cellcolor[HTML]{DAD7D7}$\delta$} & \multicolumn{1}{c|}{\cellcolor[HTML]{EBEBEB}$n_{\text{min}}$} & \multicolumn{1}{c|}{\cellcolor[HTML]{EBEBEB}Error} & \multicolumn{1}{c|}{\cellcolor[HTML]{EBEBEB}$n_{\text{min}}$} & \multicolumn{1}{c|}{\cellcolor[HTML]{EBEBEB}Error} \\ \hline
1.6                                                    & 695                                                           & 3.3135                                             & 6000                                                           & 0.3445                                             \\
3.1                                                    & 302                                                           & 6.1250                                             & 6000                                                           & 0.4728                                             \\
4.5                                                    & 171                                                           & 10.2414                                             & 6000                                                           & 0.5202                                             \\
6.0                                                    & 203                                                           & 9.3907                                            & 6000                                                           & 0.9549                                            
\end{tabular}
\caption{Comparison of the iteration index at which the smallest reconstruction error is achieved and the corresponding reconstruction error}
\label{table:eq-morozov-comparison}
\end{table}
\color{black}

When comparing the CPU execution time of these two frameworks for different noise levels, one can notice that for small noise levels, the execution times are comparable. For example, when $\delta = 1.6$, the execution times are $0.56$ seconds for the equality constrained case and $0.68$ seconds for the Morozov regularization. 

Since the convergence rate in Theorem \ref{main-theo-part} is formulated in terms of the Bregman distance $D_R$ involving the ergodic iterates, Figure \ref{fig:comparison-ergodic-different-noise-levels} displays the evolution of $D_R(\bar{u}^n, u\de)$ appearing in the statement of this theorem, for different noise levels $\delta$. The dots on each curve represent the points at which the Bregman distance reaches its minimum for each algorithm, given that the algorithm stops after reaching the maximum number of iterations, that is, $6000$.

For completeness, the analogous comparison for the current iterates, $D_R(u^n, u\de)$, is included in Figure \ref{fig:comparison-current-different-noise-levels} of Appendix \ref{appendix:figures}; however, this quantity does not appear in the theoretical result. Figure \ref{fig:comparison-sc-eq} illustrates that the source condition is indeed verified (see Appendix \ref{appendix:figures}).

\begin{figure}[H] 
    \centering
    \includegraphics[width=1.0\textwidth]{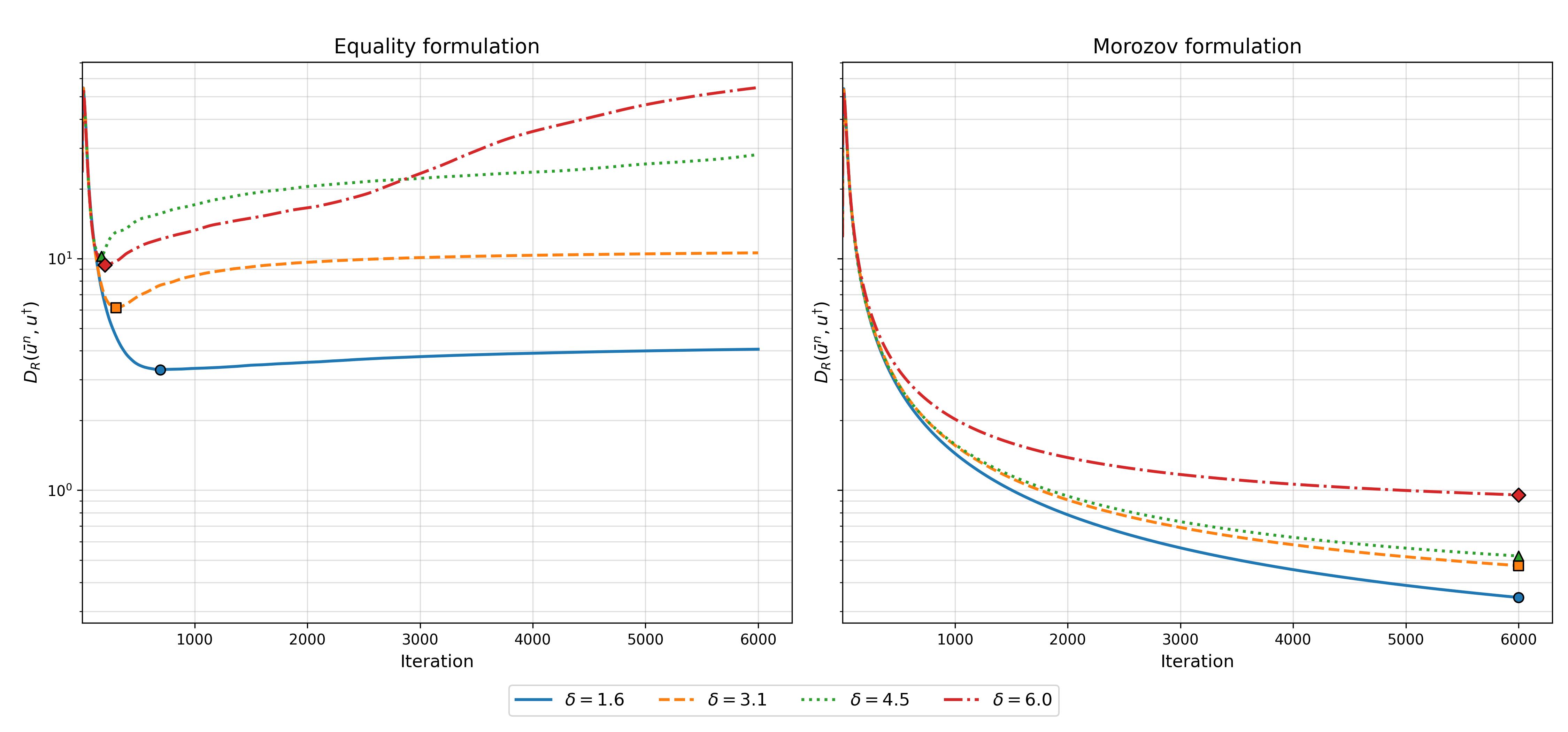}
    \caption{Comparison  of the error estimate for the ergodic sequences at different noise levels}
    \label{fig:comparison-ergodic-different-noise-levels}
\end{figure}

\subsubsection{Comparison of the basic PDHG with the accelerated PDHG method}
This subsection focuses on the problem
\begin{equation}
    \inf\limits_{u\in\mathbb{R}^{100}}\left\{ \|u\|_1 + \frac{\gamma}{2}\|u\|_2^2\right\} \text{ such that } \op u=f^{\delta},
\end{equation}
for $\gamma = 10.0$. The regularizer $R=\|\cdot\|_1 + \frac{\gamma}{2} \|\cdot\|_2^2$ is  strongly convex and promotes sparsity, while the data fidelity is $H_{f^{\delta}}(z)=i_{\{f^{\delta}\}}(z)$ with   $\delta = 0.1$. We compare the standard PDHG with the accelerated version given by Algorithm \ref{acc-alg}. 
Figure \ref{fig:acc-reconstruction} displays the reconstruction and convergence behavior of Algorithm \ref{acc-alg} in this framework.

\begin{figure}[H] 
    \centering
    \includegraphics[width=\textwidth, trim=10 5 10 5, clip]{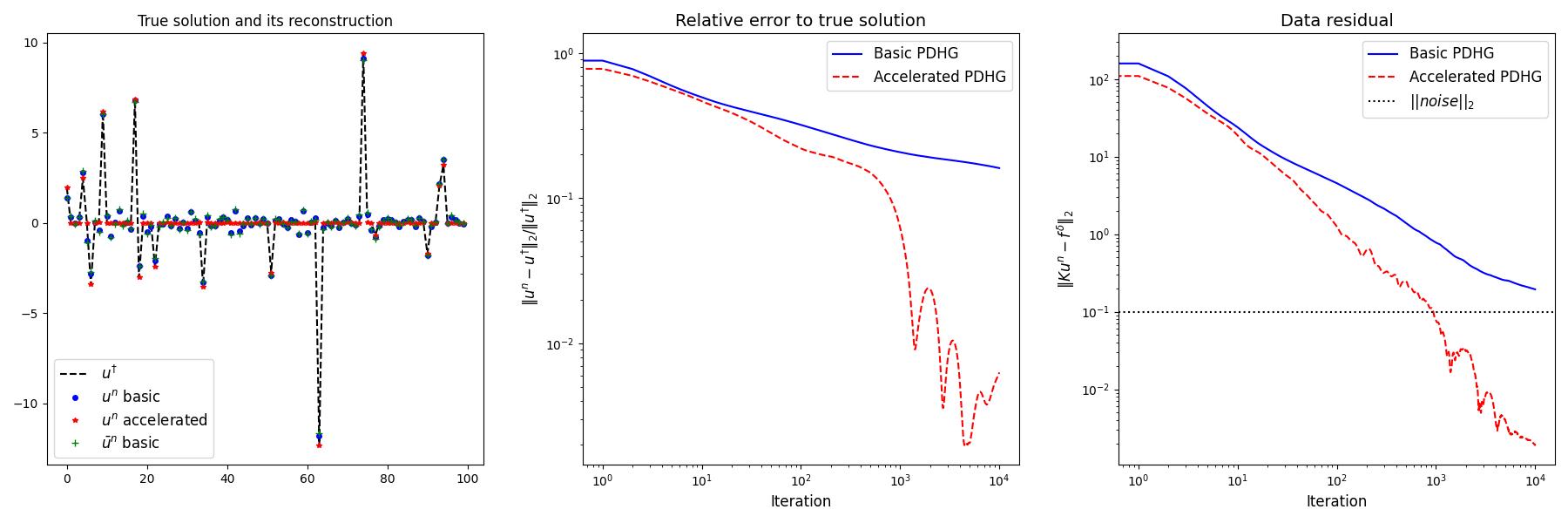}
    \caption{Reconstruction using the accelerated PDHG algorithm with equality constraint}
    \label{fig:acc-reconstruction}
\end{figure}

The plot on the left of Figure \ref{fig:comparison-conclusion-theorems} compares the error $D_R{(\bar{u}^n,u\de)}$ for the standard PDHG algorithm with the theoretical upper bound provided in the statement of Theorem \ref{main-theo-part} and confirms the expected result. The plot on the right of Figure \ref{fig:comparison-conclusion-theorems} compares the squared error for the accelerated PDHG method with the estimate \eqref{estimate-acc-alg} (dotted line) and with the upper bound \eqref{estim-seq-u} (dashed line), validating the statement of Theorem \ref{theo-acc}.

\begin{figure}[H] 
    \centering
    \includegraphics[width=\textwidth, trim=80 10 80 10, clip]{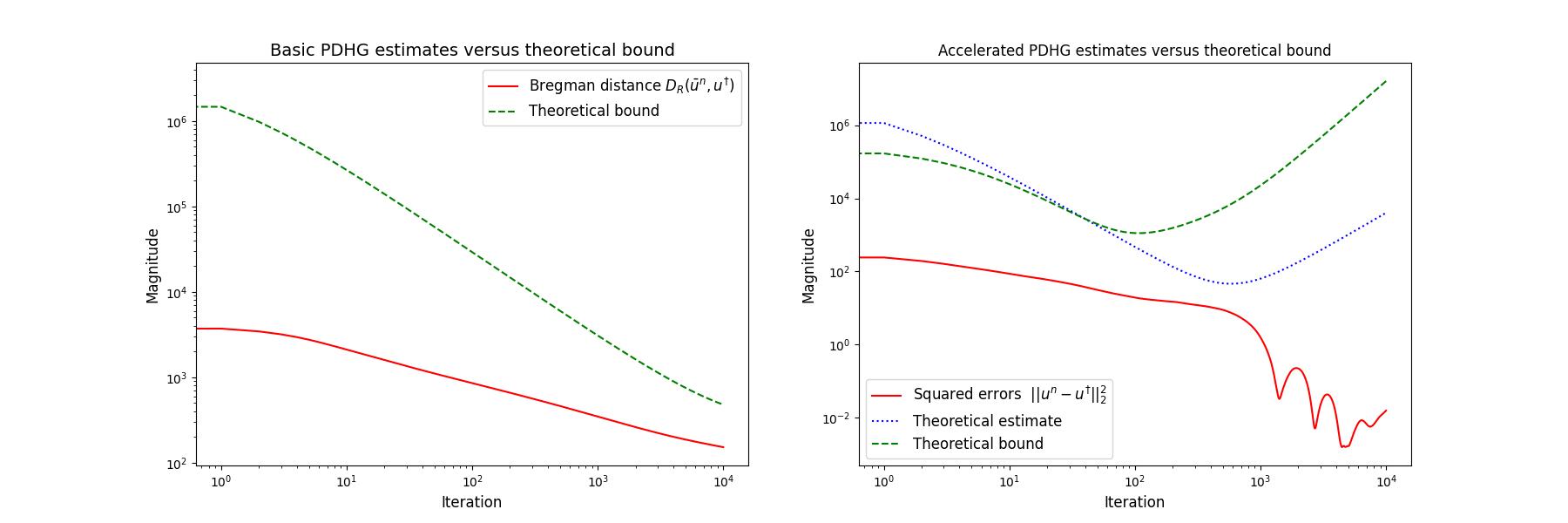}
    \caption{Comparison of the error estimates and theoretical bounds from Theorems \ref{main-theo-part} and \ref{theo-acc}}
    \label{fig:comparison-conclusion-theorems}
\end{figure}

Next, we compare the basic PDHG method with its accelerated version. However, it is not immediately clear which metric is the most appropriate for a fair comparison, since the two schemes are theoretically analyzed using different quantities. For the standard PDHG, the error is expressed in terms of the Bregman distance for the ergodic sequence $\bar{u}^n$, while for acceleration, the theoretical guaranties are formulated in terms of the squared norm for the current iterate $u^n$. 
Because of this mismatch, we present the comparison using both measures.

\begin{figure}[H] 
    \centering
    \includegraphics[width=\textwidth, trim=80 10 80 10, clip]{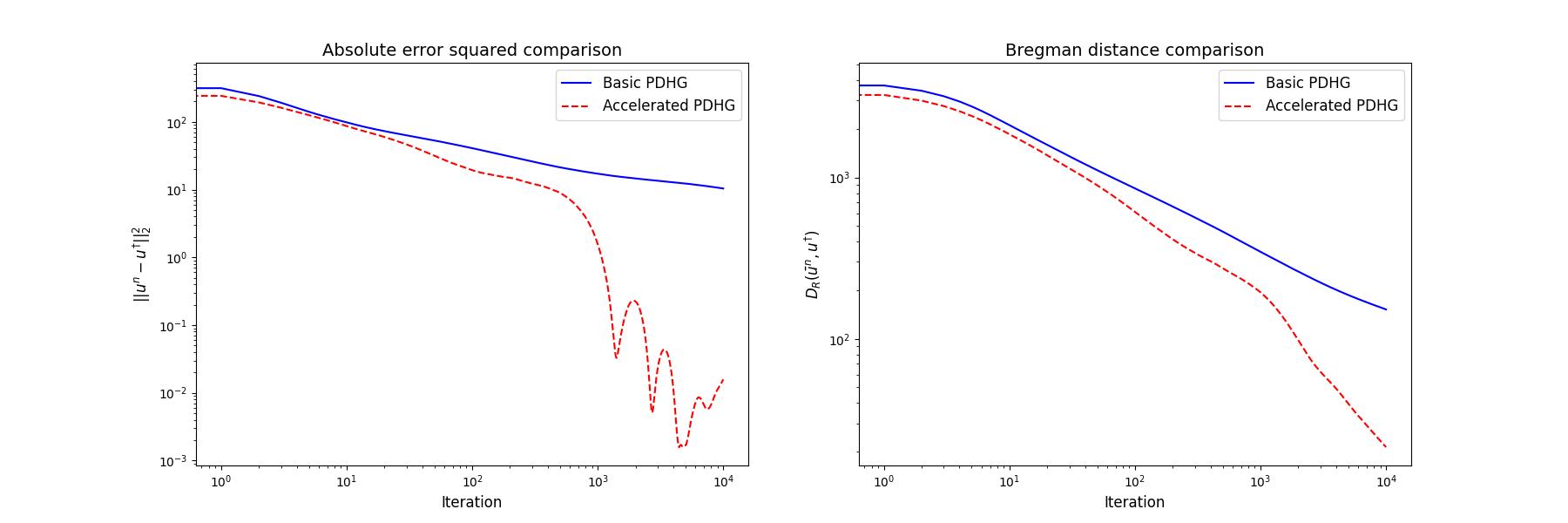}
    \caption{Comparison of basic and accelerated PDHG using two performance metrics}
    \label{fig:comparison_basic_acc}
\end{figure}
One can see in Figure \ref{fig:comparison_basic_acc} that the accelerated version of the algorithm achieves faster convergence than the basic method, showing a sharp drop followed by oscillations in the squared error measure, while maintaining a smooth decay under the Bregman distance.

\subsection{Hyperspectral abundance map recovery}
We consider hyperspectral unmixing - decomposing the mixed spectral signature at each pixel of a hyperspectral image (HSI) into pure material spectra (endmembers) and their fractional abundances - as a practically relevant instantiation of our framework in the Banach space setting of Section~\ref{ex:entropic_mirror_example}. Let $N_b$ be the number of spectral bands, $N_c$ the number of endmember classes, and $N_h \times N_w$ the spatial dimensions of the HSI. The linear mixing model reads $f^\delta = \op u^\dagger + \eta$, where $u^\dagger \in \mathbb{R}^{N_c \times N_h \times N_w}$ is the ground-truth abundance map, $f^\delta \in \mathbb{R}^{N_b \times N_h \times N_w}$ contains the noisy spectral observations, and $\eta$ is measurement noise. The forward operator $\op \colon \mathbb{R}^{N_c \times N_h \times N_w} \to \mathbb{R}^{N_b \times N_h \times N_w}$ applies the endmember signature matrix $E \in \mathbb{R}^{N_b \times N_c}$ pixel-wise, i.e., $[\op u]_{b,i,j} = \sum_{c=1}^{N_c} E_{b,c}\, u_{c,i,j}$.

Physical constraints require the abundance vector at each pixel to be non-negative and to sum to one, confining it to the $N_c$ - dimensional probability simplex
\begin{equation} \label{eq:pavia_simplex}
    \Delta^{N_c} := \left\{ u_{i,j} \in \mathbb{R}^{N_c} \;\middle|\; \sum_{c=1}^{N_c} u_{c,i,j} = 1, \; u_{c,i,j} \ge 0 \;\; \forall\, c \right\}.
\end{equation}

\paragraph{Framework instantiation}
We work in the setting of Section~\ref{ex:entropic_mirror_example}, with the specific choices
\begin{itemize}
    \item $\X = \mathbb{R}^{N_c \times N_h \times N_w}$, $\Y = \mathbb{R}^{N_b \times N_h \times N_w}$;
    \item $R(u) = i_{\Delta^{N_c}}(u)$ (per-pixel simplex indicator);
    \item $G_\alpha(u) = \frac{\alpha}{2}\|\nabla u\|_2^2$, with $\nabla$ the standard 2D forward finite-difference operator and $\alpha \in \{0,1\}$, which acts as a smooth spatial bias rather than the regularizer itself - regularization is achieved via early stopping or the Morozov discrepancy constraint;
    \item $\varphi(u) = \sum_{c,i,j}(u_{c,i,j}\log u_{c,i,j} - u_{c,i,j})$ (negative Boltzmann-Shannon entropy), so that $D_\varphi$ becomes the Kullback-Leibler divergence;
    \item $\psi(v) = \tfrac{1}{2}\langle \Sigma^{-1} v, v\rangle$ (quadratic dual reference with step-size operator $\Sigma$);
    \item $H_{f^\delta} = i_{\{f^\delta\}}$ (equality constraint) or $H_{f^\delta} = i_{C^\delta}$ with $C^\delta = \{z : \|f^\delta - z\|_2 \le \delta\}$ (Morozov discrepancy constraint).
\end{itemize}
Under the entropic reference $\varphi$ and the simplex indicator $R$, the primal mirror step reduces to a multiplicative softmax update, i.e.,
\begin{equation} \label{eq:pavia_primal_update}
    u_{c,i,j}^{k+1} = \frac{u_{c,i,j}^k \exp\bigl(-\tau [\op^* v^k + G'_\alpha(u^k)]_{c,i,j}\bigr)}{\sum_{c'=1}^{N_c} u_{c',i,j}^k \exp\bigl(-\tau [\op^* v^k + G'_\alpha(u^k)]_{c',i,j}\bigr)},
\end{equation}
which maintains simplex feasibility by construction and avoids an explicit projection step. For numerical stability, the exponent is shifted by its channel-wise maximum before exponentiation, and iterates are clamped to a floor of $10^{-15}$ before renormalization. The dual reference $\psi$ yields either a preconditioned linear update (when $\Sigma$ is diagonal) or the standard Euclidean update (when $\Sigma = \sigma I$). The resulting algorithms -Algorithms~\ref{alg:entropic_unmixing_equality} and~\ref{alg:entropic_unmixing_morozov} - are the entropic counterparts of Algorithm~\ref{Condat-Vu-al-4} and Algorithm~\ref{Condat-Vu-alg-6}, respectively.
\begin{algorithm}[H]
\caption{Entropic mirror descent for equality constrained spectral unmixing}
\label{alg:entropic_unmixing_equality}
\begin{algorithmic}[1]
\State \textbf{Input:} Initial values $u^0, v^0$, step size $\tau > 0$, dual step size operator $\Sigma$ (diagonal matrix or scalar $\sigma$), data $f^\delta$.
\For{$k = 0, 1, 2, \dots$}
\State Compute the primal update for each pixel $(i, j)$ and channel $c$:
\begin{equation*}
    u_{c, i, j}^{k+1} = \frac{u_{c, i, j}^k \exp\left( -\tau [ \op ^* v^k + G'_\alpha(u^k) ]_{c, i, j} \right)}{\sum_{c'=1}^{N_c} u_{c', i, j}^k \exp\left( -\tau [ \op ^* v^k + G'_\alpha(u^k) ]_{c', i, j} \right)}
\end{equation*}
\State Compute the dual update:
$
    v^{k+1} = v^k + \Sigma (\op (2u^{k+1} - u^k) - f^\delta)
$
\EndFor
\State \textbf{return} $u^{k+1}, v^{k+1}$
\end{algorithmic}
\end{algorithm}

\begin{algorithm}[H]
\caption{Entropic mirror descent for Morozov-constrained spectral unmixing}
\label{alg:entropic_unmixing_morozov}
\begin{algorithmic}[1]
\State \textbf{Input:} Initial values $u^0, v^0$, step size $\tau > 0$, dual step size $\sigma > 0$, data $f^\delta$, noise level $\delta$.
\For{$k = 0, 1, 2, \dots$}
\State Compute the primal update for each pixel $(i, j)$ and channel $c$:
\begin{equation*}
    u_{c, i, j}^{k+1} = \frac{u_{c, i, j}^k \exp\left( -\tau [ \op ^* v^k + G'_\alpha(u^k) ]_{c, i, j} \right)}{\sum_{c'=1}^{N_c} u_{c', i, j}^k \exp\left( -\tau [ \op ^* v^k + G'_\alpha(u^k) ]_{c', i, j} \right)}
\end{equation*}
\State Compute the auxiliary dual variable:
$
    z^k = v^k + \sigma (\op (2u^{k+1} - u^k) - f^\delta)
$
\State Compute the dual update via soft-thresholding:
\begin{equation*}
    v^{k+1} = \begin{cases} 0 & \text{if } \|z^k\|_2 \le \sigma \delta, \\ \left(1 - \frac{\sigma \delta}{\|z^k\|_2}\right) z^k & \text{if } \|z^k\|_2 > \sigma \delta \end{cases}
\end{equation*}
\EndFor
\State \textbf{return} $u^{k+1}, v^{k+1}$
\end{algorithmic}
\end{algorithm}
\clearpage
\begin{figure}[tbp]
    \centering
    \includegraphics[width=0.5\textwidth]{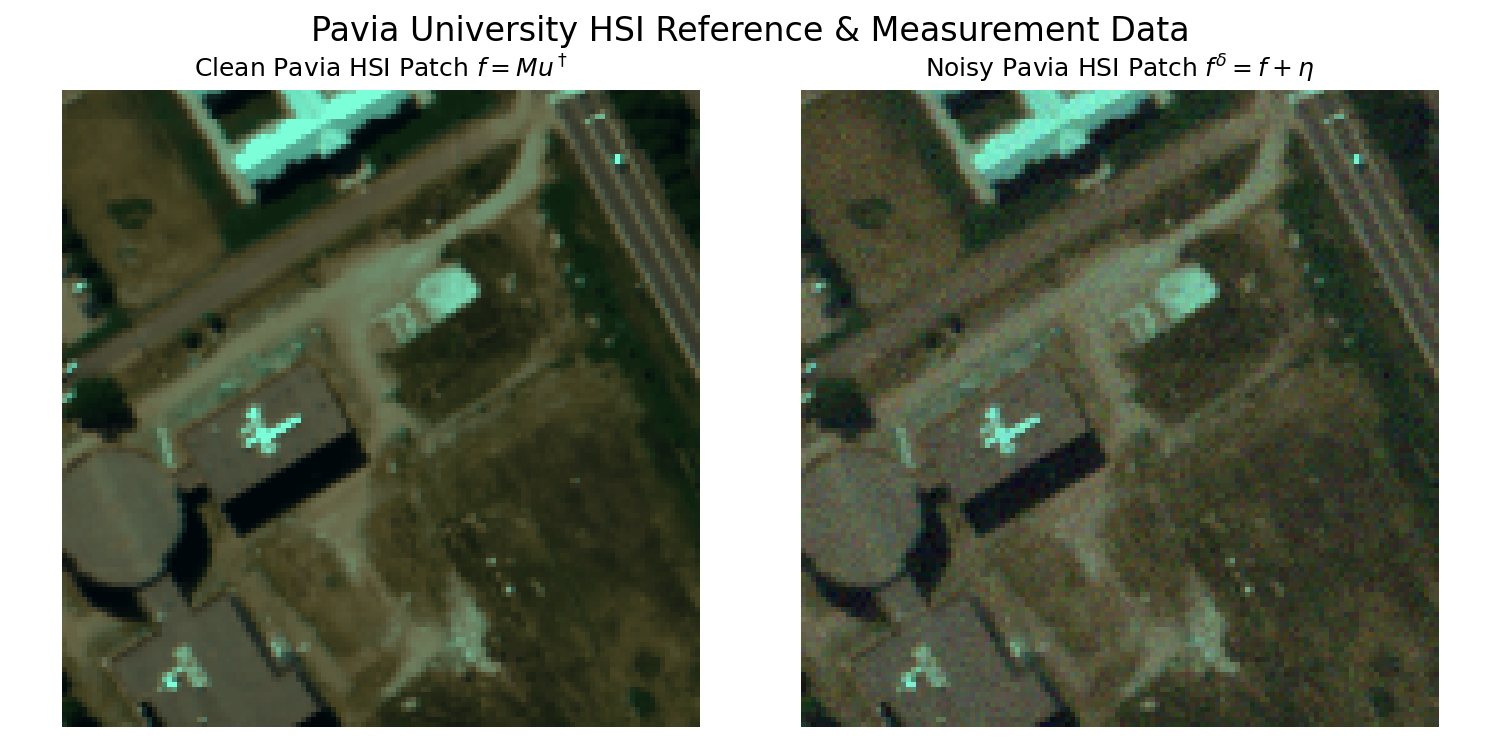}
    \caption{Left: clean spatial patch $f$. Right: noisy spatial patch $f^\delta$ with 1\% Gaussian noise. Both panels are false-color RGB composites (bands 55, 30, 5).}
    \label{fig:pavia_reference_data}
    \vspace{10pt}
    \includegraphics[width=1.0\textwidth]{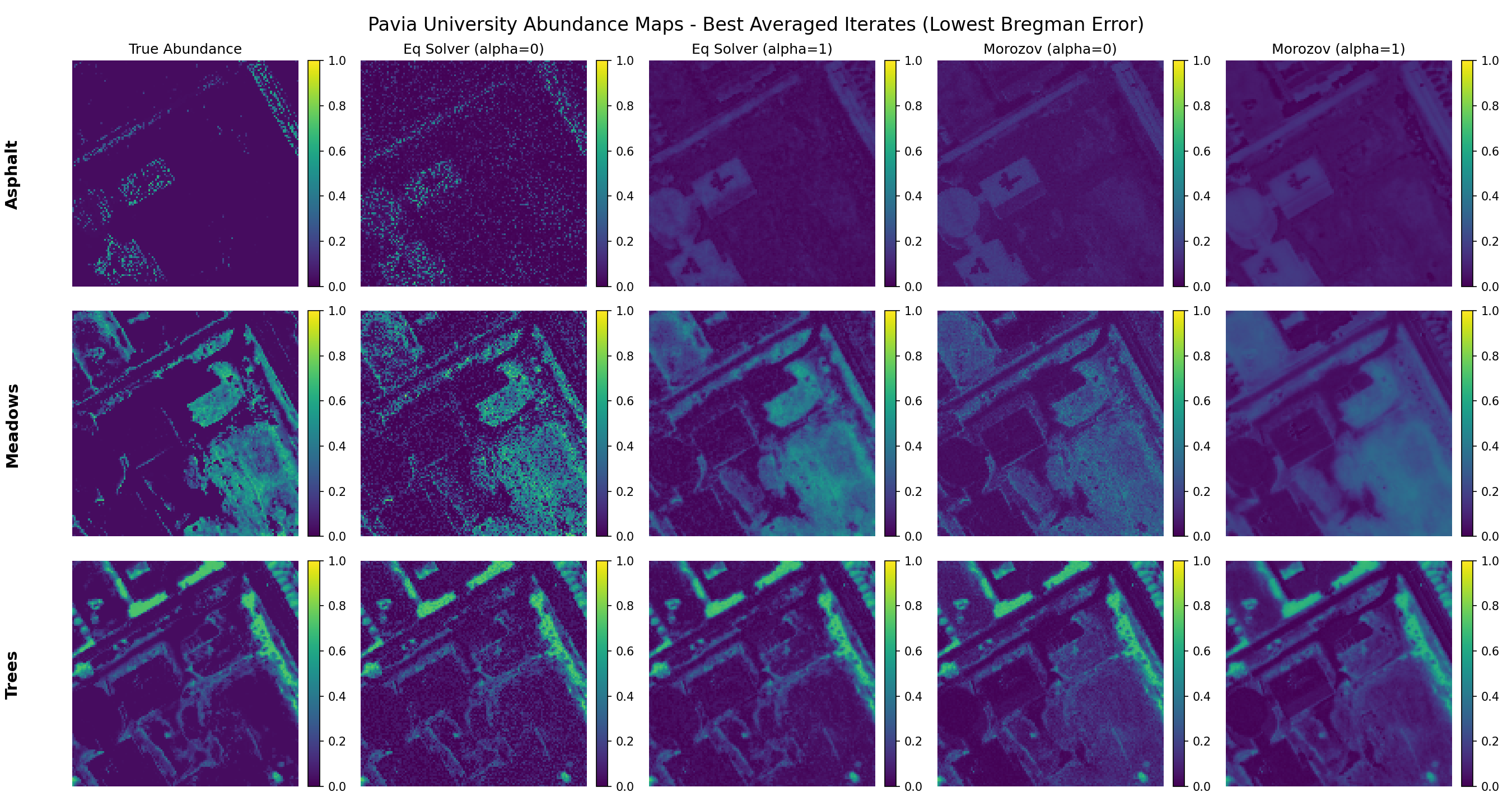}
    \caption{Abundance maps (best averaged iterates, minimizing $D_{R+G_\alpha}$). Rows: Asphalt, Meadows, Trees. Columns: ground truth, equality solver ($\alpha = 0$, $\alpha = 1$), Morozov solver ($\alpha = 0$, $\alpha = 1$).}
    \label{fig:pavia_abundances_comparison_avg}
    \vspace{10pt}
    \includegraphics[width=1.0\textwidth]{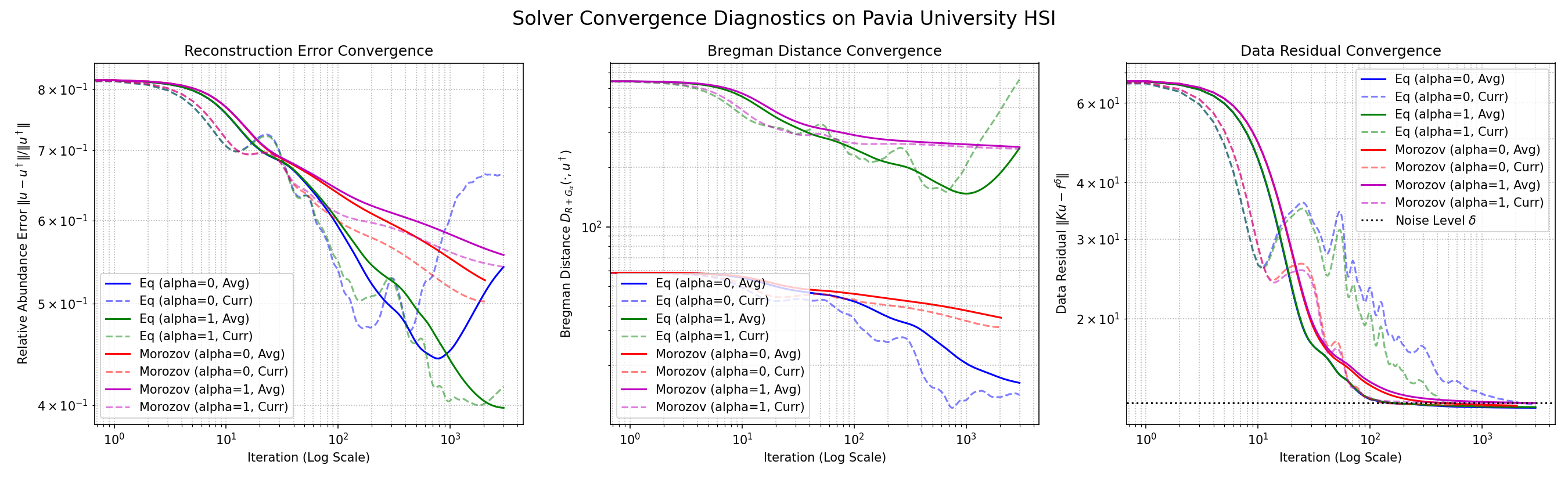}
    \caption{Left: relative error $\|u - u^\dagger\|_2/\|u^\dagger\|_2$. Middle: $D_{R+G_\alpha}(\cdot, u^\dagger)$. Right: $\|\op u - f^\delta\|_2$ (dotted line: noise level $\delta$). Solid: averaged iterates; dashed: current iterates; horizontal axis: log scale.}
    \label{fig:pavia_convergence_curves}
\end{figure}
\clearpage

\paragraph{Connection to the error bounds}
Because $R = i_{\Delta^{N_c}}$ and both $\bar{u}^n, u^\dagger \in \Delta^{N_c}$, the Bregman distance from Theorem~\ref{main-th-gen} simplifies to
\begin{equation} \label{eq:pavia_bregman_simplification}
    D_{R+G_\alpha}(\bar{u}^n, u^\dagger) = G_\alpha(\bar{u}^n) - G_\alpha(u^\dagger) + \langle \op^* v^\dagger, \bar{u}^n - u^\dagger \rangle,
\end{equation}
where $v^\dagger$ denotes the optimal dual variable from the source condition. In the present setting, we have $p = q = 2$ (via Pinsker's inequality for the primal $\ell^1$-norm and the Euclidean structure of the dual space), so the noise propagation term in the bound~\eqref{estim:non-acc-gen} grows as $O(\delta^2 n)$. This linear growth explains the  behavior observed in both formulations: The iterates initially approach $u^\dagger$ but then deteriorate as noise is amplified. Under equality constraints, the error eventually diverges. Under the Morozov constraint, the soft-thresholding of the dual variable bounds the noise propagation and the error plateaus near the regularized solution $u^*$ satisfying $\|\op u^* - f^\delta\| \le \delta$. In both cases, early stopping remains beneficial.

\paragraph{Experimental setup}
We use a $128 \times 128$ spatial patch of the Pavia University HSI dataset~\cite{83h4-0556-25}, which contains $N_b = 103$ spectral bands and $N_c = 9$ material classes. The endmember signature matrix $E \in \mathbb{R}^{N_b \times N_c}$ is formed by averaging the globally normalized spectral profiles over all annotated pixels of each class; the resulting matrix has condition number $\approx 1812$ under the spectral norm.

A ground-truth abundance map $u^\dagger$ is obtained by running Algorithm~\ref{alg:entropic_unmixing_equality} for $20{,}000$ iterations on the clean patch with diagonal dual preconditioning ($\sigma_b = \sigma_0 / \sum_{c} E_{b,c}$, $\sigma_0 = 1$), which accelerates convergence for the ill-conditioned $E$. The primal step size is chosen to satisfy the inequality $\tau_{\mathrm{pre}} \max_c \sum_b \sigma_b E_{b,c}^2 < 1$.  After adding a baseline offset of $0.05$ and renormalizing (so that $u^\dagger$ lies strictly in the simplex interior), we generate noisy data $f^\delta = \op u^\dagger + \eta$ with $1\%$ additive Gaussian noise and record the noise level $\delta = \|f^\delta - f\|_2$.

For the comparison runs ($3000$ iterations on $f^\delta$), diagonal preconditioning is no longer available because the Morozov dual projection lacks a closed-form solution when $\Sigma$ is non-scalar. We therefore set $\Sigma = \sigma I$ with $\sigma = 0.5$ and choose
\begin{equation} \label{eq:pavia_step_size}
    \tau = \frac{0.99}{\sigma \max_c \sum_{b=1}^{N_b} E_{b,c}^2 + \alpha L},
\end{equation}
where $L = 8.0$ is the Lipschitz constant of $G_1'$, in accordance with the step-size condition of Lemma~\ref{lem-str-conv}. We compare four configurations: equality and Morozov constraints, each with $\alpha = 0$ (no spatial bias) and $\alpha = 1$.

\paragraph{Results}
Figure~\ref{fig:pavia_reference_data} shows the reference and noisy data. The reconstructed abundance maps for the averaged iterates (at the iteration minimizing the Bregman distance error) are displayed in Figure~\ref{fig:pavia_abundances_comparison_avg} for three representative material classes; the corresponding current iterates are provided in Figure~\ref{fig:pavia_abundances_comparison_curr} of Appendix~\ref{appendix:figures}. The convergence diagnostics - relative error $\|u - u^\dagger\|_2/\|u^\dagger\|_2$, Bregman distance $D_{R+G_\alpha}(\cdot, u^\dagger)$, and data residual $\|\op u - f^\delta\|_2$ are plotted in Figure~\ref{fig:pavia_convergence_curves}. The equality constrained runs exhibit classical semi-convergence, while the Morozov-constrained runs stabilize once the discrepancy reaches $\delta$. Including the smooth bias ($\alpha = 1$) visibly reduces noise artifacts in the recovered abundance maps.

\section*{Conclusions}
In this work, we establish a regularization theory for certain primal–dual splitting algorithms applied to ill-posed inverse problems with noisy data. In the Hilbert space setting, we derive error estimates for both the Condat--Vũ and accelerated PDHG methods under a standard source condition. We carry over the analysis of the non-accelerated Condat--Vũ algorithm to Banach spaces and develop an entropic PDHG instance in a natural non-reflexive setting. The proposed framework allows for a general class of convex data fidelities, with its assumptions explicitly verified for both equality-constrained problems and Morozov regularization. Last but not least, numerical experiments on a sparse recovery problem and hyperspectral abundance map recovery confirm the theoretical results and demonstrate the applicability of the considered methods.

While our analysis is restricted to problems affected by additive noise, extending the framework to non-additive noise models is an interesting direction for future research. Moreover, investigating accelerated schemes with other relevant choices of step sizes remains an open problem.

\section*{Disclosure statement}
The authors report there are no competing interests to declare.

\section*{Declaration of generative AI use}
During the preparation of this manuscript, the authors used ChatGPT-5.5, Gemini 3.6 and Grammarly Pro  for linguistic revision, and assistance with code development for the numerical experiments. The authors edited and verified all text and experimental code, and take full responsibility for the content of the final work.

\section*{Acknowledgements}
The research of Diana-Elena Mirciu and Elena Resmerita was funded in whole or in part by the Austrian Science Fund (FWF) 10.55776/PAT4084924.

\appendix
\section{Appendix}\label{appendix:proofs}

\subsection{Proof of Theorem \ref{lemma-Schmidt-gen}}\label{appendix:proof-aux-result}
\begin{proof}
    Let $n\geq1$ be fixed. The assumptions of the theorem imply that for any $k\in\{1,\dots, n\}$, 
    $
        a_k^q \le S_k + \sum_{i=1}^k \lambda_i a_i.
    $
    Let us denote by $b_n:=\max\{a_k:1\leq k\leq n\}$. The sequence $(S_n)$ is non-decreasing and $(\lambda_n)\subset[0,+\infty)$, thus
    $
        a_k^q \le S_n + b_n\sum_{i=1}^k \lambda_i. 
    $
    From the non-negativity of the sequences $(a_n)$ (which implies that $(b_n)\subset[0,+\infty)$) and $(\lambda_n)$, one finds
    $
        a_k^q \le S_n + b_n\sum_{i=1}^n \lambda_i. 
    $
    The previous inequality holds for $k$ arbitrarily fixed in $\{1,\dots, n\}$, therefore
    $
        b_n^q \le S_n + b_n\sum_{i=1}^n \lambda_i. 
    $
    Now we use Young inequality in the form $ab\leq \frac{a^p}{p} + \frac{b^q}{q}$ for the second term on the right-hand side to derive
    $
        b_n^q \le S_n + \frac{b_n^q}{q} + \frac{\left(\sum_{i=1}^n \lambda_i\right)^p}{p}, 
    $
    or, equivalently,
   $
        \frac{q-1}{q}b_n^q \le S_n + \frac{\left(\sum_{i=1}^n \lambda_i\right)^p}{p}. 
    $
    This implies 
    $
        a_n\leq b_n \leq \left(\frac{q}{q-1}S_n + \frac{q}{q-1} \frac{\left(\sum_{i=1}^n \lambda_i\right)^p}{p}\right)^{1/q},
    $
    where the definition of $b_n$ gives the first inequality.
    Note that $\frac{q}{(q-1)p}=1$, which establishes the conclusion.
\end{proof}
\subsection{Additional figures - Comparison of the equality constrained formulation with the Morozov regularization}\label{appendix:figures}

\begin{figure}[H]
    \centering
    \includegraphics[width=0.8\textwidth]{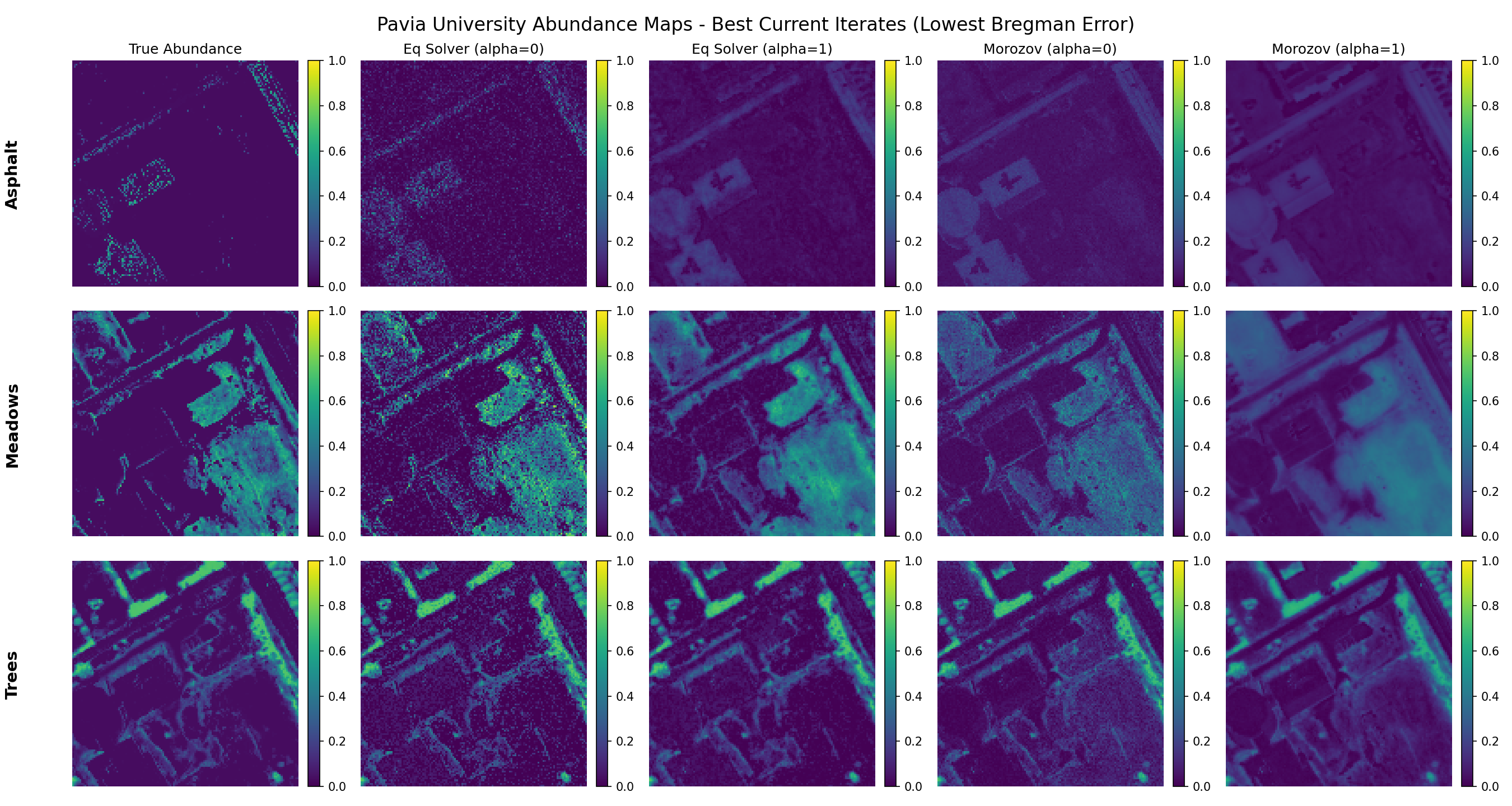}
    \caption{Abundance maps (best current iterates, minimizing $D_{R+G_\alpha}$). Rows: Asphalt, Meadows, Trees. Columns: ground truth, equality solver ($\alpha = 0$, $\alpha = 1$), Morozov solver ($\alpha = 0$, $\alpha = 1$).}
    \label{fig:pavia_abundances_comparison_curr}
\end{figure}

\begin{figure}[H] 
    \centering
    \includegraphics[width=1.0\textwidth]{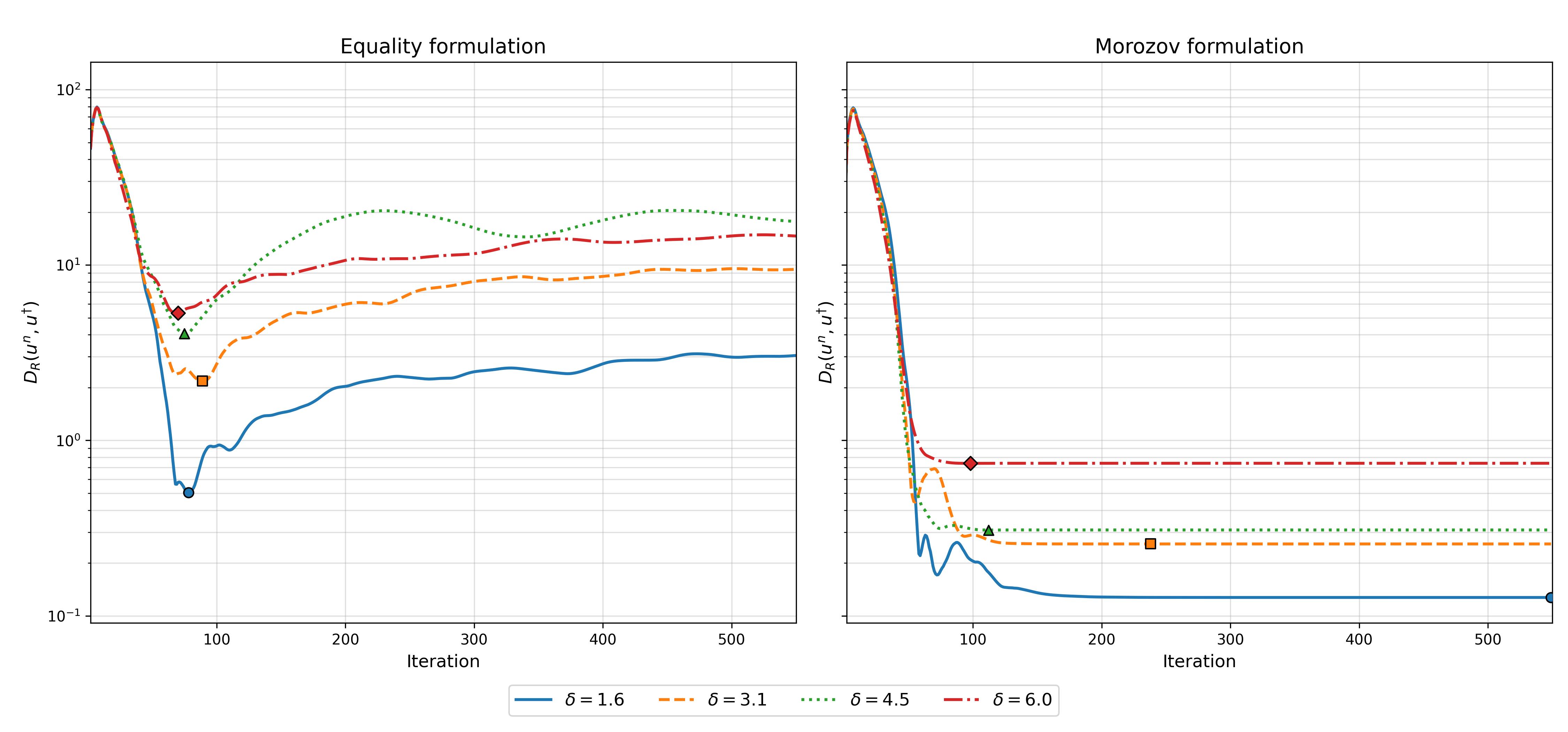}
    \caption{Comparison of the error estimate for the current iterates at different noise levels}
    \label{fig:comparison-current-different-noise-levels}
\end{figure}

\begin{figure}[H] 
    \centering
    \includegraphics[width=0.7\textwidth]{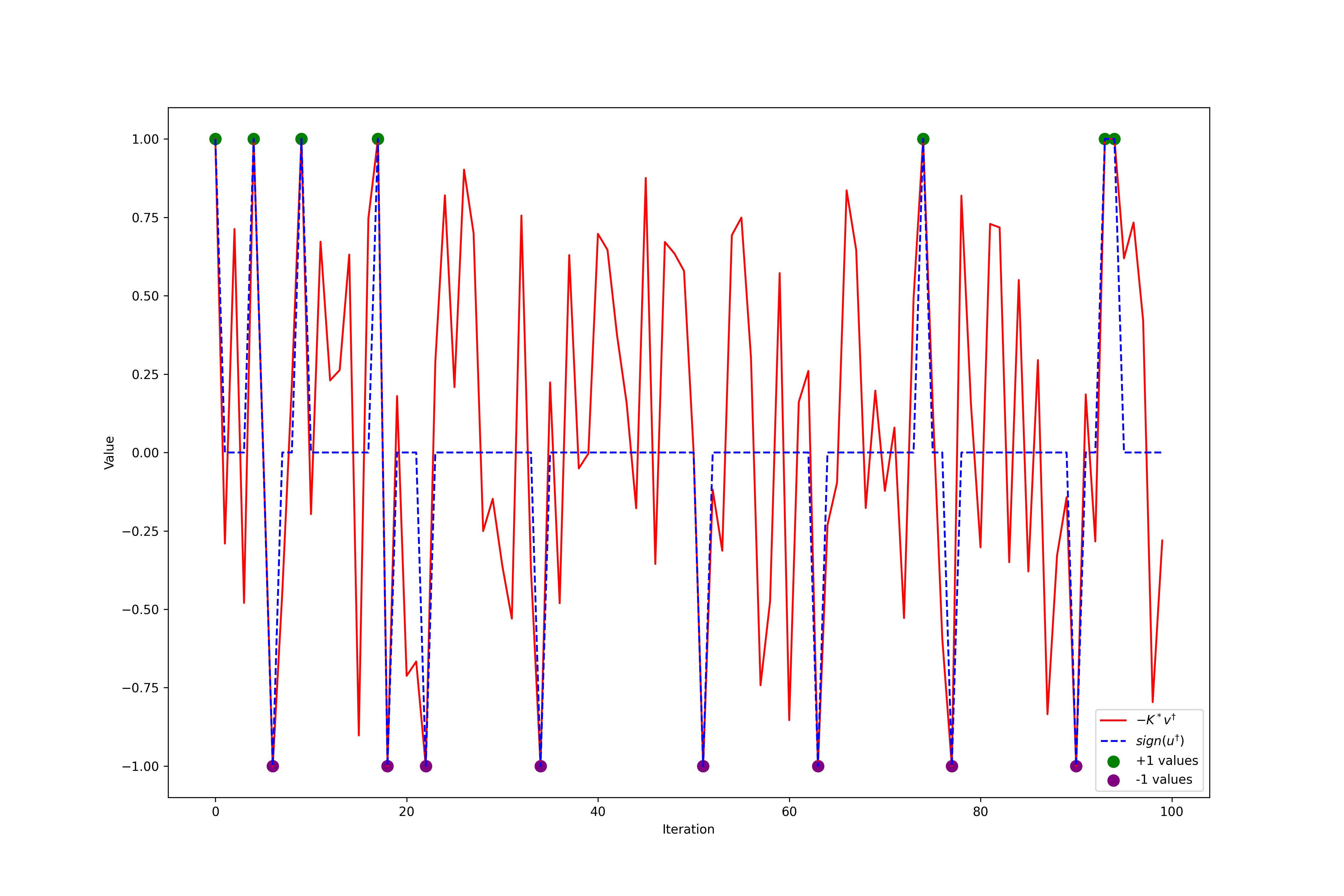}
    \caption{Comparison between $-\op ^*v\de$ and the sign of $u\de$, indicating that the estimated source element $v\de$ is correct.}
    \label{fig:comparison-sc-eq}
\end{figure}

\printbibliography
\end{document}